\documentclass[a4paper]{scrartcl}
\usepackage{amssymb,amsmath,amsthm,authblk,natbib,dsfont}

\usepackage{times}
\usepackage[plain,noend]{algorithm2e}
\usepackage{epsfig}
\usepackage{tabularray}
\usepackage{tikz}
\usepackage{url}

\usepackage[top=2cm,left=2cm]{geometry}  %% just for draft; by Stephan
\title{Smooth Backfitting for Additive Hazard Rates} 
\author[1]{Stephan M.~Bischofberger\footnote{Corresponding author: Stephan M.~Bischofberger, e-mail: bischofberger@staburo.de, address: Aschauer Str. 26a, 81549 Munich, Germany}}
\author[2]{Munir Hiabu}
\author[3]{Enno Mammen}
\author[4]{Jens Perch Nielsen}
\affil[1]{Staburo GmbH,  Munich, Germany}
\affil[2]{Department of Mathematical Sciences, University of Copenhagen, Denmark}
\affil[3]{Institute for Applied Mathematics, Heidelberg University, Germany}
\affil[4]{Bayes Business School, City St George’s, University of London, United Kingdom}
\let\originalleft\left
\let\originalright\right
\renewcommand{\left}{\mathopen{}\mathclose\bgroup\originalleft}
\renewcommand{\right}{\aftergroup\egroup\originalright}

\makeatletter
\renewcommand{\algocf@captiontext}[2]{#1\algocf@typo. \AlCapFnt{}#2} % text of caption
\def\@algocf@capt@plain{top}
\renewcommand{\algocf@makecaption}[2]{%
  \addtolength{\hsize}{\algomargin}%
  \sbox\@tempboxa{\algocf@captiontext{#1}{#2}}%
  \ifdim\wd\@tempboxa >\hsize%     % if caption is longer than a line
    \hskip .5\algomargin%
    \parbox[t]{\hsize}{\algocf@captiontext{#1}{#2}}% then caption is not centered
  \else%
    \global\@minipagefalse%
    \hbox to\hsize{\box\@tempboxa}% else caption is centered
  \fi%
  \addtolength{\hsize}{-\algomargin}%
}
\makeatother

\DeclareMathOperator*{\argmin}{arg\,min} % by Stephan
\DeclareMathOperator{\trace}{trace} % by Stephan
\DeclareMathOperator{\Var}{Var} % by Stephan
\newtheorem{remark}{Remark}

\newtheorem{proposition}{Proposition}
\newtheorem{theorem}{Theorem}
\newtheorem{lemma}{Lemma}

\allowdisplaybreaks  % allow page breaks in align

\begin{document}

\maketitle

\begin{abstract}
Smooth backfitting was first introduced in an additive regression setting via a direct projection alternative to the classic backfitting method by Buja, Hastie and Tibshirani. 
This paper translates the original smooth backfitting concept to a survival model considering an additively structured hazard. The model allows for censoring and truncation patterns occurring in many  applications such as medical studies or actuarial reserving. Our estimators are shown to be a projection of the data into the space of multivariate hazard functions with smooth additive components. Hence, our hazard estimator is the closest nonparametric additive fit even if the actual hazard rate is not additive. This is different to other additive structure estimators where it is not clear what is being estimated if the model is not true.
We provide full asymptotic theory for our estimators.
We  propose an implementation of estimators that show good performance in practice. 
\end{abstract}

\section{Introduction}
This paper introduces a fundamental model and estimator for structured multivariate marker dependent hazards: the smooth backfitting of additive hazards. In structured non-parametric regression,   \cite{Mammen:etal:99} modelled and estimated the additive structure by projecting data onto the appropriate additive subspace. The resulting projection estimator is known as the smooth backfitting estimator. The name comes from the fact that when calculating the projection estimator iteratively, then one must not only  smooth the component that is being updated, but all components.  This is different to classical backfitting \citep{Buja:etal:89} where only the component that is  being updated is smoothed. It has been shown that smooth backfitting performs much better than previous comparable smoothing kernel based backfitting approaches, in particular in high dimensional problems and with correlated covariates, see  \cite{Nielsen:Sperlich:05}. A theoretical comparison between classical and smooth backfitting for additive regression models was recently done in \citep{Huang:Yu:19}, explaining why smoothing of all components has a better adaption.
 Since the initial smooth backfitting paper many variations and extensions have been developed using smooth backfitting to tackle more sophisticated problems in mathematical statistics,  \cite{mammen2003generalised},  \cite{yu2008smooth}, \cite{mammen2009nonparametric}, \cite{mammen2014additive}, \cite{han2018smooth}, \cite{mammen2022backfitting}, \cite{Bissantz:etal:16}, \cite{han2019additive}, \cite{jeon2020additive}, \cite{hiabu2021smooth} and \cite{gregory2021statistical}.

The aim of the current paper is to transfer the original approach of additive non-parametric structures to marker dependent hazard estimation and to allow for a potentially high number of covariates with possibly correlated markers. It turns out that when the original estimation problem is phrased as a minimisation problem in the correct way via a counting process formulation, then our smooth backfitting additive hazard approach can be implemented and analysed in a very similar way to smooth backfitting in regression. We see this as a necessary step to understand more complicated structures in marker dependent hazards. 
The additive subspace is closed making analysis more accessible and the additive structure allows for a  more immediate interpretation than more complicated models of structured hazards.
One important alternative structure is the multiplicative or proportional hazard model.  Survival analysis  practitioners often work with such multiplicative marker dependent hazard models, including the Cox model. Smooth backfitting for the  multiplicative model was recently analysed in \cite{hiabu2021smooth}, where the analysis was challenged by the shape of the multiplicative subspace that is not closed like the additive subspace is and where some tricks had to be developed, e.g. a  solution weighted optimization,  to arrive at a tractable estimation method and analysis. The additive approach developed in this paper does not face  these two latter challenges and it had perhaps been more natural to have developed this current paper first and then \cite{hiabu2021smooth} afterwards. Both this current paper and \cite{hiabu2021smooth} arrive at the same conclusion for smooth backfitting of marker dependent hazard estimators as the authors in \cite{Nielsen:Sperlich:05} did for smooth backfitting of non-parametric regression: Smoothing all components in every iteration step and not only smoothing the component that is  being updated is important. Otherwise the estimator breaks down in many cases - in particular in high dimensions - where smooth backfitting still works. Smooth backfitting seems more reliable than classical backfitting of kernel estimators and we expect that the additive marker dependent hazard model and estimator of this paper can be an important starting point for further developments of structured marker dependent hazard approaches in survival analysis, just like the many developments we have seen in non-parametric regression.
Code to replicate our simulation and application can be found under \url{https://github.com/MHiabu/Replicate-Smooth-Backfitting-for-Additive-Hazard-Rates}.
In the next section we give some insight on the additive model itself and its role in marker dependent hazard models as a practical survival analysis tool.

\section{Additive structured hazards and related literature}

One well known model in hazard regression is the proportional hazards model of \cite{Cox:72} and it is has been seen as the natural equivalent to additive regression functions in linear and nonparametric regression. 
As pointed out in \cite[p.~103]{Martinussen:Scheike:06}, additive hazard models  have been ``somewhat overlooked in practice'' although they share the same advantages of additive regression models concerning both theoretical properties and implementation. To the best knowledge of the authors, this is still the case, with some exceptions \citep{%Martinussen:Vansteelandt:13,
Tchetgen:etal:15,Aalen:etal:19,Dukes:etal:19%,Ying:etal:19
}. However, in certain applications an additive relationship in the hazard function is indeed more plausible than a proportional one \citep{Beslow:Day:87,Lin:Ying:94,Kravdal:97,McDaniel:etal:19}. 
Moreover, \cite[pp.155f]{Aalen:etal:08} provides a variety of reasons for additive risk factors. 

In the original additive hazards model  \citep{Aalen:80}, %(also Aalen:89,Aalen:93)
the intensity of a counting process $\{N(t):t\in[0,1]\}$, conditional on the $d$-dimensional covariate $Z(t)=(Z_1(t), \dots, \allowbreak Z_d(t))^T$, satisfies 
\begin{equation}\label{eq:simple:additive:model}
\lambda(t) =  Z^T(t)\beta(t) Y(t)
\end{equation}
at time $t$ with a regression coefficient $\beta(t)=(\beta_1(t),\dots,\beta_d(t))^T$ and exposure $Y$ which is equal to unity when an individual is at risk. An overview about this model is given in \cite{Martinussen:Scheike:06} in which the authors praise it as a simple nondistributional model that is easy to implement. Nonparametric estimators of the cumulative regression coefficient $B(t) =\int_0^t \beta(s) \mathrm d s$ in model \eqref{eq:simple:additive:model} have been examined in \cite{McKeague:88} and \cite{Huffer:McKeague:91} among others.  
%The authors focus on the estimation of the cumulative regression coefficient $B(t) =\int_0^t \beta(s) \mathrm d s$ and prove uniform $n^{1/2}$-rate convergence of their estimator $\hat B$. 
%Based on a weighted least squares estimator for the cumulative hazard functions $A_j(t) = \int_0^t \alpha_j(s) \mathrm ds$, asymptotic theory for an estimator of $\alpha$ is derived in \cite{Huffer:McKeague:91}.

Model \eqref{eq:simple:additive:model} imposes a linear relationship between the intensity and the value of the covariates through $Z^T(t)\beta(t) $. 
%This is often too restrictive and does not reflect the data. 
%\begin{center} \textbf{find reference for this} \end{center}
%Hence, 
We  loosen up the assumption of linearity. 
Before introducing the model we investigate in this article, we describe the most general model and its disadvantages and explain why we assume certain additive constraints. The completely nonparametric conditional intensity model 
\begin{equation} \label{eq:unstructured}
\lambda(t) = \alpha(t|Z)  Y(t)
\end{equation}
for a conditional hazard function $\alpha$ generalizes model \eqref{eq:simple:additive:model} making it the most flexible model. As it is common, we assume 
$ 
\alpha(t|Z) = \alpha(t,Z(t))
$  in this paper, i.e.\ that the conditional hazard at time $t$ given the covariates only depends on the values of the covariates at time $t$ and not on the values of the past. 
%The conditional hazard is then defined via
%\[
% \alpha(t|Z)=\lim_{h \downarrow 0} h^{-1}\mathbb{P}\left(T\in [t, t+h) | \ T\geq t, \{Z(s), s\leq t\}\right).
%\]

Model \eqref{eq:unstructured} has first been introduced for time-constant covariates in \cite{Beran:81}. Time dependent covariates were considered in \cite{McKeague:Utikal:90} and \cite{Nielsen:Linton:95}. Other examples from the vast literature on nonparametric hazard estimators for this model include \cite{VanKeilegom:Veraverbeke:01} or \cite{Spierdijk:08}. % The latter investigates bandwidth selection %Although it is noted that the generalization to higher dimensions is feasible, it is not given. 
%and for further theory about bandwidth selection (cross-validation and so-called double one-sided cross-validation) for unstructured conditional hazard rates see \cite{Gamiz:etal:16}. 
Without further structural restrictions, estimators of \eqref{eq:unstructured} suffer from the curse of dimensionality: The rate of convergence decreases exponentially.This is a well known issue for unstructured nonparametric estimators, making them in many cases  in-practical already in dimensions higher than, say, three.
That one can not do better in the unstructured nonparametric case is known at least since \cite{Stone:80} who provided formulas for the best possible rate of convergence for nonparametric estimators. Accordingly, the aforementioned nonparametric hazard  estimators were only illustrated for the case with one-dimensional covariate $Z$. 

To overcome this issue, one has to focus on a model that is more restrictive than the unstructured nonparametric hazard model \eqref{eq:unstructured}. We restrict our assumptions on an additive model which is nested in \eqref{eq:unstructured}. However, instead of the original additive Aalen model \eqref{eq:simple:additive:model}, we assume that the hazard rate consists of additive nonparametric components, 
\begin{equation}\label{eq:additive:hazard}
\alpha(t,  z) = \alpha^* + \alpha_0(t) + \alpha_1(z_1) + \dots + \alpha_d(z_d),
\end{equation} 
with smooth, but not further restricted,  components $\alpha_k$, $k=1,\dots d$, depending on  covariate values $z_1,\dots,z_d$.  The constant $\alpha^*$ is a norming constant making the decomposition unique, as will later be further specified.
%  In the sequel and throughout this paper, we refer to model \eqref{eq:additive:hazard} as the additive model and we introduce the concept of smooth backfitting for it. 

The additive model \eqref{eq:additive:hazard} is both more general but also more restrictive than the additive Aalen model \eqref{eq:simple:additive:model}.
It is more restrictive because it does not allow  the effect of covariates on the hazard to change with time. It is  more general because the effect of the covariates on the hazard do not need to be linear.
A very interesting model that  generalises both models is to replace each component $\alpha_k(z_k), k\geq 1,$ in  \eqref{eq:additive:hazard}  by a two-dimensional components $\alpha_k(t, z_k)$
capable of capturing a covariate effect that changes  with time. While we do not consider this more general  setting in this paper, we see the work done in this paper as a crucial step towards developing methods of such a more general kind. Another possible generalisation is to consider multiple time scales, see e.g. \cite{hiabu2021nonsmooth}.

To estimate the components in   \eqref{eq:additive:hazard},  we propose a  local polynomial least squares minimisation under the constraint \eqref{eq:additive:hazard}. The solution can be identified with the projection of the observation into the space of local polynomial additive hazard functions and can be calculated through a simple iterative procedure. 
We call the resulting estimator additive smooth backfitting hazard estimator. 

When estimating the hazard function  $\alpha(t,  z)$, by  the nature of equation \eqref{eq:additive:hazard},  it can happen that the  estimate is negative at certain points.
This is expected to happen especially more if the underlying hazard function is far from being additive.  However, it is reassuring that the smooth backfitting components $\hat \alpha_k$ will still have a clear interpretation as approximation of the closest additive fit. In practice, if probabilities need to be calculated, one ad-hoc solution is to  use the non-additive  adjusted hazard
\[
\alpha^{adj}(t,z)= \max(\alpha(t,z),\varepsilon), \varepsilon\geq 0.
\]
Indeed, this is also what we do in the application Section \ref{sec:cprs} for $\varepsilon=0$ with satisfying results.

\section{The additive hazard model} \label{sec:model5}
Let $\mathcal{T} >0$. We observe $n$ \textit{i.i.d.}~copies of the stochastic processes $\{(N(t),Y(t),Z(t)) : t\in [0,\mathcal{T}]\}$ where $N$ is a right-continuous counting process which is zero at time zero and which has jumps of size one. We assume that $Y$ is a left-continuous stochastic process with values in  $\{0,1\}$ and which equals unity if the observed individual is at risk. Moreover, let $Z$ be a $d$-dimensional left-continuous stochastic process with $Z(t) \in [0,R]^d$, $t\in[0,\mathcal{T}]$, for some $R>0$. The multivariate process $((N_1,Y_1,Z_1), \dots, \allowbreak (N_n,Y_n,Z_n))$ is assumed to be adapted to the filtration $\{\mathcal F_t:t\in[0,\mathcal{T}]\}$ which satisfies the \emph{usual conditions} \citep[p.~60]{Andersen:etal:93}. 

In the following, we assume that for each $i=1,\dots,n$, the process $N_i$ satisfies Aalen's multiplicative intensity model, i.e. that its intensity $\lambda_i$ satisfies 
\begin{equation} \label{eq:aalens:model}
\lambda_i(t)=\lim_{h \downarrow 0} h^{-1}\mathbb E[N_i((t+h)-)-N_i(t-)|\ \mathcal F_{t-}]=\alpha(t,Z_i(t))Y_i(t),
\end{equation}
where $Y_i(t)$ is indicating if individual $i$ is at risk at time $t$. 
The function $\alpha(t,Z(t))$ is the conditional hazard rate given the covariates $Z$ at time $t$. Furthermore, we assume that $\alpha$ satisfies the additive structure of model \eqref{eq:additive:hazard}, which we write as
\[
\alpha(t,Z_i(t)) = \alpha^* + \sum_{j=0}^d \alpha_j(X_{ij}(t))
\]
with the notation $X_i(t)=(t,Z_{i1}(t),\dots,Z_{id}(t))\in \mathcal X$ for $\mathcal X =[0,\mathcal{T}]\times [0,R]^d$. In the sequel, we will also write $x=(t,z_1,\dots,z_d)\in \mathcal X$ and henceforth $\alpha(x)=\alpha(t,z)$ for short.

Each component of the additive hazard $\alpha$ is only identifiable up to an additive shift. Later, we will give conditions under which each component is uniquely identified. 

%\subsection{Left truncation and right censoring time as covariates} 
%\begin{center} \textbf{check if too much copied from SB for multiplicative hazard from here on} \end{center}

Model \eqref{eq:aalens:model} allows for different kind of filtered data making it very flexible. 
These filterings include left-truncation and right-censoring which occurs in many applications of survival analysis \citep{Martinussen:Scheike:06}. %the application in Section \ref{sec:application}.
We now illustrate how to embed left-truncated covariates and right-censored survival time into model \eqref{eq:aalens:model}. 
Let $T$ denote the survival time. 
Left-truncation means that we observe copies of $(T,Z)$ only on a compact subset $\mathcal I \subseteq \mathcal X$ with the property that 
 $(t_1,Z(t_1))\in \mathcal I$  and  $t_2\geq t_1$ imply $(t_2,Z(t_2)) \in \mathcal I$ almost surely.
We allow $\mathcal I$ to be random but assume it is independent from $T$ given $Z$. 
The survival time $T$ can also be subject to right censoring with censoring time $C$ as long as $C$ is conditionally independent from $T$ given the covariate process $Z$. 
This condition holds in particular if the censoring time equals one of the components of $Z$. 
Hence, under this filtering scheme, we observe $n$ \textit{i.i.d.}\ copies of $(\widetilde T , Z^{*}, \mathcal I,\delta)$, where 
$\delta=\mathds 1 (T^{*} <C), \ \widetilde T= \min (T^{*},C)$, and $(T^{*}, Z^{*})$ is the truncated version of $(T,Z)$,  i.e, $(T^*, Z^*)$  arises from $(T,Z)$ by conditioning on the event  $\{(T,Z(T))\in \mathcal I\}$.

We can now define a counting process $N_i$ for each individual $i=1,\dots,n$, via 
\[N_i(t)=\mathds 1 \left\{\widetilde T_i \leq t,\ \delta_i=1\right\}, 
\] with respect to the filtration
$\mathcal F_{i,t}=\sigma \left( \bigg\{\widetilde T_i\leq s,\  Z^   *_i(s), \ \mathcal I_i, \ \delta_i : \ s\leq t\bigg\} \cup \mathcal N\right),$
for a class of null-sets $\mathcal N$, which completes the filtration. 
In this setting it can be easily shown that, under above assumption of $\alpha(t|Z)=\alpha(t,Z(t))$, Aalen's multiplicative intensity model \eqref{eq:aalens:model} is satisfied with hazard rate 
\begin{align*}
\alpha(t,z) =\lim_{h \downarrow 0} h^{-1}\mathbb P\left(   T_i\in [t, t+h)| \   T_i\geq t, \ Z_i(t)=z\right), 
\end{align*}
and exposure
\begin{align*}
Y_i(t)&=  \mathds 1 \big\{(t,Z^*_i(t))\in \mathcal I_i, \ t\leq \widetilde T_i\big\},
\end{align*}
for individual $i$. 
The sets $\mathcal I_i$ are allowed to be independent random copies of $\mathcal I$.

\section{The smooth backfitting estimator of additive hazards} \label{sec:estimators5}

\subsection{Smooth backfitting hazard estimator as projection} \label{appendix:projection}

In this and the next section we illustrate the equivalence of projections and estimators that minimize squared errors following the line of \cite{Mammen:etal:99} where smooth backfitting was first introduced for nonparametric regression.  The idea of describing smoothing estimators as projections in a regression setting  is explained in great detail in \cite{Mammen:etal:01}. 
In the following we  introduce this projection principle for a counting process framework. 

We will introduce our estimators as a projection from a functional space $\mathcal H$ onto a certain subspace. The choice of the subspace, implies the class of functions that can be estimated and also the class of estimators to be considered. We now  specify these functional spaces as well as (semi-)\linebreak[0]{}norms.

We define the unrestricted functional space as
\[
\mathcal H = \{ (f^{i,j})_{i =1 , \dots, n, j = 0, \dots, d+1} ;  f^{i,j}:\mathbb{R}^{d+2} \to \mathbb{R}  \}, 
\]
and subsets $\mathcal H_{full}^{LC} \subseteq  \mathcal H_{full}^{LL} \subseteq \mathcal H$ via 
\begin{align*}
\mathcal H_{full}^{LL} = \{ f \in \mathcal H : &f^{i,j}(s,x) \text{ does not depend on } i,s\}, \\
\mathcal H_{full}^{LC} = \{ f \in \mathcal H : &f^{i,j}(s,x) \text{ does not depend on } i,s; \\
&f^{i,j}(s,x) \equiv 0 \text{ for } j=1,\dots, d+1\}.
\end{align*}

Furthermore, for additive hazard functions we define additive subsets 
\begin{align*}
\mathcal H_{add}^{LL} = \{ f \in \mathcal H_{full}^{LL} : &f^{i,0}(s,x)  = \sum_{j=0}^d g_j(x_j) ;  %\\  &
f^{i,j}(s,x) = h_j(x_j)  ,\  j=1,\dotsc,d+1, \\
&\text{for some functions } g_j, h_j:\mathbb R\to \mathbb R  \},  \\
\mathcal H_{add}^{LC} = \{ f \in \mathcal H_{full}^{LC}: &f^{i,0}(s,x)  = \sum_{j=0}^d g_j(x_j) \text{ for some functions } g_j :\mathbb R\to \mathbb R  \},
\end{align*}
that contain the class of local linear and local constant hazard estimators, respectively. 
Moreover, we define a semi-norm $\lVert \cdot \rVert$ on $\mathcal H$ through
\begin{align*}
%\lVert f \rVert^2 = \int \frac 1 n \sum_{i=1}^n \left[ f^{i,0}(x) + \sum_{j=1}^d f^{i,j} (x) \frac{x_j-X_j^i}{h}\right ]^2 \prod_{j=1}^d K_h(X_j^i - x_j) \mathrm dx 
\lVert f \rVert^2 = \int \int \frac 1 n \sum_{i=1}^n &\left[ f^{i,0}(s,x) + \sum_{j=0}^{d} f^{i,j+1} (s,x) \left(\frac{x_j-X_{i,j}(s)}{h}\right) \right ]^2  \\
&\times   Y_i(s) K_h(x- X_i(s)) \mathrm ds \, \mathrm d\nu(x), 
\end{align*}
for $f\in\mathcal H$ and where $\nu$ is a measure with strictly positive density. % and for brevity we  write  $\lVert \cdot \rVert$ for  its induced norm $\lVert \cdot \rVert_\mathcal G$ on every subspace $\mathcal G \subseteq \mathcal H$. \\
This semi-norm will be used to define the projection in the sequel. 

Next we will illustrate how $\mathcal H $ contains both hazard functions and the observations $(N_i)$, $i=1,\dots,n$.  
For every $\varepsilon >0$, the data can be identified with an element $\Delta_\varepsilon N \in \mathcal H$ via 
\[
\Delta_\varepsilon N^{i,0}(s,x) = \frac 1 \varepsilon \int_s^{s+\varepsilon} \mathrm  dN_i(s), \ \ \ \ \Delta_\varepsilon N^{i,j}(s,x) \equiv 0, \ \ \ \ \ j=1,\dotsc,d+1. 
\]

We define the unstructured local constant and local linear hazard estimator  as 

\begin{align}\label{eq:unstructred:estimator}
\lim_{\varepsilon \to 0} \argmin_{\theta\in\mathcal H_{full}^{LC}} \lVert \Delta_\varepsilon N - \theta \rVert, \ \ \ \ \ 
\lim_{\varepsilon \to 0} \argmin_{\theta\in\mathcal H_{full}^{LL}} \lVert \Delta_\varepsilon N - \theta \rVert,
\end{align}

respectively.
One can easily verify that these estimators coincide with the well known local constant and local linear hazard marker dependent hazard estimators introduced in \cite{Nielsen:Linton:95} and \cite{Nielsen:98b}.

For $\varepsilon \to 0$, each element $\Delta_\varepsilon N^{i,0}$ converges to a Dirac delta function. Hence, we write 
\[
\min_{\theta\in\mathcal G} \lVert \Delta N - \theta \rVert := \lim_{\varepsilon \to 0} \min_{\theta\in\mathcal G} \lVert \Delta_\varepsilon N - \theta \rVert,
\]
for $\mathcal G\subseteq \mathcal H$. 

We define the local constant and local linear nonparametric additive hazard estimator respectively as
\begin{align}\label{eq:proj:estimator}
 \argmin_{\theta\in\mathcal H_{add}^{LC}} \lVert \Delta N - \theta \rVert, \quad 
 \argmin_{\theta\in\mathcal H_{add}^{LL}} \lVert \Delta N - \theta \rVert.
\end{align}

For the minimisation over all additive hazard functions, we can either use a direct projection into $\mathcal H^P_{add}$, $P\in\{LC,LL\}$ which is given by $\min_{\theta\in \mathcal H^P_{add}} \lVert \Delta N - \theta \rVert $  or we use a Pythagorean argument to project in two steps: For $\hat \alpha \in \mathcal H^P_{add}$, it holds $ \lVert \Delta N - \hat\alpha \rVert^2= \lVert \Delta N -\tilde \alpha   \rVert^2 + \lVert  \tilde \alpha - \hat\alpha \rVert^2$ with  $\tilde \alpha\in\mathcal H_{full}^P$. The last identity holds because the elements  $\Delta N - \tilde \alpha$ and $ \tilde \alpha - \hat\alpha $ are orthogonal \citep{Mammen:etal:01}.  In additive marker dependent hazard estimation, the unrestricted marker dependent hazard estimators can be understood as intermediate  in an iterative projection procedure that first projects to   the unrestricted space and then to the additive space.

\subsection{Smooth backfitting hazard estimator via least squares} \label{subsec:introduce:estimators}

In the previous section, we introduced the local constant estimator as a projection from $\mathcal H$.
In this section, we show how this connects to the more known least squares criteria, and thereby also state the estimator in a way that is more directly mathematically tractable.
We first consider the unstructured local polynomial hazard estimators. For a general understanding, we write down the general formulation for polynomials of order $p$, but in this paper we will only consider the local constant and the local linear case, $p=0,1$.

We will estimate the additive components of the hazard function via kernel smoothers. Let
$k:\mathbb R \to \mathbb R$ be a symmetric and continuous kernel function such that $\int k(u) \mathrm du =1 $.
We define  $K(u_0,\dotsc,u_d)=\prod_{j=0}^d k(u_j)$. For a   smoothing parameter $h>0$,
$K_h(u)=\prod_{j=0}^d k_h(u_j)=\prod_{j=0}^d h^{-1}k(h^{-1}u_j)$. 
In the sequel, we will use a modification of the kernel function to ensure that the kernel always integrates to unity. 
We replace $k_h(u-v)$ by
\begin{equation} \label{eq:boundedkernel}
k_h(u,v) = I_{(u,v \in [0,1])}\left( \int k_h(s-v)\mathrm ds \right)^{-1} k_h(u-v)
\end{equation}
for every $h>0$ to correct for normalization at the boundaries from now on. %This ensures $\int k_h(u,v)\mathrm dv = 1$ for all $u\in[0,1]$. 
Furthermore, we define the multivariate kernel 
\[
 K_h(u,v) = \prod_{j=0}^d   k_h(u_j,v_j), 
\]
for $u=(u_0,\dots,u_d)$ and $v=(v_0,\dots,v_d)$.

The unstructured $p$th order local polynomial estimator of the hazard function in $x$ is defined as the first component of
\begin{align}
\begin{split} \label{eq:criterion}
\lim_{\varepsilon \to 0} \argmin_{\substack{\theta_0: \mathbb R^{d+1} \to \mathbb R \\ \theta_j:\mathbb R^{d+1} \to \mathbb R^{d+1} \\ j=1,\dots,p}} \sum_{i=1}^n \int \int  &\left \{ \frac 1 \varepsilon \int _s^{s+\varepsilon} \mathrm dN_i(u) - \theta_0(x) \right.\\
&-\theta_1^T(x)\left(\frac{x_0 - X_{i0}(s)}{h}, \dots, \frac{x_d - X_{id}(s)}{h}  \right)^T  - \dotsc \\
& \left.  - \theta_{p}^T(x)\left(\left(\frac{x_d - X_{id}(s)}{h}\right)^p, \dots,\left(\frac{x_d - X_{id}(s)}{h}\right)^p \right)^T\right\}^2 \\ &\times K_h(x,X_i(s)) Y_i(s) \mathrm ds\, \mathrm) d\nu(x) ,
\end{split} % \notag
\end{align}

The cases $p=0,1$ are exactly the local constant and local linear projection estimator defined in \eqref{eq:unstructred:estimator}.

For the rest of this paper, we limit ourselves to the same kernel $k$ and bandwidth $h$ for each dimension to keep the notation simple. 
Henceforth, if there is no confusion about the boundaries of the integrals, $\int$ denotes integration over the whole support $[0,\mathcal{T}]\times[0,R]^d$. %or any rectangular subsets of lengths $\mathcal{T}$ or $R$, respectively. 
The measure $\nu$ has to have a strictly positive density but the estimator does not depend on the specific choice of $\nu$ if we don't have restrictions on the functions $\theta_j$.
We will specify a weighting function $w$ such that $\mathrm d\nu(x) = w(x) \mathrm dx$. 
Note that this estimator allows for local polynomial approximation at degree $p$ but it is not additive yet. 

The nonparametric additive hazard estimator we investigate in this paper is defined by the minimisation in equation \eqref{eq:criterion} under the following constraints on the structural form of $\theta$. 
For $p=0$,  the constraint $\theta_0(x) = \bar \alpha ^* + \sum_{j=0}^d \bar \alpha_j(x_j)$ for some functions $\bar \alpha_0,\dots, \bar \alpha_d$ and a constant $\bar \alpha^*$, leads to the local constant estimator as introduced in \eqref{eq:proj:estimator}:
\begin{align}
\begin{split} \label{eq:criterion:LC}
 \lim_{\varepsilon \to 0} \argmin_{\substack{\bar\alpha^* \in \mathbb R, \\ \bar \alpha_j:\mathbb R \to \mathbb R, \\ j = 0,\dots,d}} \sum_{i=1}^n \int \int  &\left \{ \frac 1 \varepsilon \int _s^{s+\varepsilon} \mathrm dN_i(u) - \left[ \bar \alpha^* +\bar \alpha_0(t) + \bar\alpha_1(z_1) + \dots \bar \alpha_d(z_d) \right ] \right \}^2 \\ 
& \times K_h(x,X_i(s)) Y_i(s) \mathrm ds \, \mathrm d\nu(x). 
\end{split} 
\end{align}
For the unique identification of the constant component $\alpha^*$ and the components $\alpha_j$, $j=0,\dots,d$, we will set further constraints in equation \eqref{eq:cond}. 
%Furthermore, we will take a weighting function $w$ such that $\mathrm d\nu(x) = w(x) \mathrm dx$ which is to be specified later. We will establish existence and uniqueness of the estimator. 

The local linear additive hazard estimator as defined in \eqref{eq:proj:estimator} arises by setting $\theta_0(x) = \bar \alpha ^* + \sum_{j=0}^d \bar \alpha_j(x_j)$ and $\theta_1(x)=\left(\partial / \partial x_0 \theta_0(x), \ldots, \partial / \partial x_d \theta_0(x)\right)$.
\begin{align}
\begin{split} \label{eq:criterion:LL}
 \lim_{\varepsilon \to 0} \argmin_{\substack{ \bar \alpha^* \in \mathbb R, \\ \bar \alpha_j:\mathbb R \to \mathbb R, \\ \bar \alpha'_j:\mathbb R \to \mathbb R,  \\ j = 0,\dots,d}} \sum_{i=1}^n \int \int  &\left \{ \frac 1 \varepsilon \int _s^{s+\varepsilon} \mathrm dN_i(u) -  \Big[ \bar \alpha^* + \bar \alpha_0(t) +\bar \alpha_1(z_1) + \dots \bar \alpha_d(z_d)   \right.  \\ 
& \ +  \left. \bar \alpha_0'(x_0)\left(\frac{x_0 - X_{i0}(s)}{h}\right)+ \dots  + \bar \alpha'_d(x_d)\left(\frac{x_d - X_{id}(s)}{h}\right) \Big] \right \}^2  \\
& \times K_h(x,X_i(s)) Y_i(s) \mathrm ds \,\mathrm d\nu(x). 
\end{split} 
\end{align}
Existence and uniqueness of the minimizers of \eqref{eq:criterion:LC} and \eqref{eq:criterion:LL} will be established later. 

\subsection{The local constant smooth backfitting additive kernel hazard estimator}\label{sec:loc:const}
The minimisation  in equation \eqref{eq:criterion} for $p=0$ leads to the unstructured  local constant estimator $\hat\alpha^{LC}$ defined via $\hat \alpha^{LC}(x)= \hat O(x) / \hat E(x) $ with
\begin{align*}
\hat O(x) &= \frac 1 n   \sum_{i=1}^n \int   K_h(x,X_i(s))\mathrm dN_i(s), \\
\hat E(x) &=  \frac 1 n \sum_{i=1}^n \int  K_h(x,X_i(s)) Y_i(s) \mathrm ds.
%,\\ \kappa_n(x) &= \left(  n \int K_h(x-u) \mathrm du  \right)^{-1}.
\end{align*}
for $x\in\mathcal X$. 
The estimators $\hat O$ and $\hat E$ estimate the occurrence and exposure of the observations. The exposure $E$ is defined via $E(x)= f_t(z) \mathbb E[Y(t) ]$ where $f_t(z)$ is the conditional density of $(Z_1(t),\dots,Z_d(t))$ given $Y(t)=1$. The occurrence is defined as $O(x)=\alpha(x)E(x)$ for $x=(t,z)\in \mathcal X$. The structure of a hazard estimator as an estimator of occurrence divided by an estimator of exposure is in line with piece-wise constant hazard estimators in \cite{Martinussen:Scheike:02}. 

To define the local constant smooth backfitting additive hazard estimators we proceed as follows. 
%In this chapter, we just write for short $\hat E= \hat E^{LC}$ and $\hat O = \hat O^{LC}$.
Following the derivation in Section \ref{subsec:introduce:estimators}, the estimator is defined through equation \eqref{eq:criterion:LC}. 
The solution $\bar\alpha=(\bar\alpha^*,\bar\alpha_0,\dots,\bar\alpha_d)$ satisfies the first order conditions
\begin{equation}\label{eq:alphastar}
\bar \alpha^* = \frac{ \int_{\mathcal{X}} [\hat \alpha^{LC}(x)-\sum_{j=0}^d \bar\alpha_j(x_j)]w(x)\mathrm dx}{\int_{\mathcal{X}}w(x)\mathrm dx}
\end{equation}
and
%\[
%\bar \alpha_k(x_k) = \frac{ \int_{\mathcal{X}_{x_k}}  [\tilde \alpha(x) - \bar \alpha^* - \sum_{j\neq k} \bar \alpha_j(x_j)]w(x) \mathrm  dx_{-k}}{\int_{\mathcal{X}_{x_k}}  w(x) \mathrm dx_{-k}} 
%,\]
%which can be written as
\begin{equation}\label{eq:backfit}
\bar \alpha_k(x_k) = \int_{\mathcal{X}_{x_k}} \hat \alpha^{LC}(x) \frac{w(x)}{w_k(x_k)} \mathrm dx_{-k} - \sum_{j\neq k} \int_{\mathcal{X}_{x_k}} \bar\alpha_j(x_j) \frac{w(x)}{w_k(x_k)}  \mathrm dx_{-k}  -  \bar\alpha^*
,\end{equation}
for $k=0,\dots,d$, where we write $w_k(x_k) = \int_{\mathcal{X}_{x_k}}  w(x) \mathrm  dx_{-k}$ for the marginals of $w$ using the notation $\mathcal{X}_{x_k} = \{ y \in \mathcal{X} : y_x = x_k\}$ and $dx_{-k}$ denoting integration over all components except for $k$. %c.f. equation (8) in \cite{Mammen:etal:99}.
For the unique identification of the solution we also set the conditions 
\begin{equation}\label{eq:cond}
\int_{\mathcal X_{k}}  \bar \alpha_k(x_k) w_k(x_k) \mathrm dx_k = 0, \ \ \ \ \ k=0,\dotsc,d.
\end{equation}

%Note that this results in the components $\bar \alpha_k$ being negative for some values of $x$. The additive factor $\bar \alpha^*$ adjusts $\bar \alpha$ making it non-negative. 
These identification conditions enable us further to get 
\begin{align*}
\bar \alpha^* %&= \frac{ \int_{\mathcal{X}} \tilde \alpha(x) w(x)\mathrm  dx- \sum_{j=0}^d \int_{\mathcal X_{x_j}} \bar\alpha_j(x_j) \int_{\mathcal X_{-j}} w(x)\mathrm dx_{-j}\mathrm  dx_j}{\int_{\mathcal{X}}w(x)\mathrm dx} \\
%&= \frac{ \int_{\mathcal{X}} \tilde \alpha(x) w(x)\mathrm dx - \sum_{j=0}^d \int_{\mathcal X_{x_j}} \bar\alpha_j(x_j) w_j(x_j) \mathrm dx_j}{\int_{\mathcal{X}}w(x )\mathrm dx} \\
&= \frac{ \int_{\mathcal{X}} \hat \alpha^{LC}(x) w(x)\mathrm dx}{\int_{\mathcal{X}}w(x)\mathrm dx}
= \frac { \int_{\mathcal{X}} \hat O(x) \mathrm dx } {\int_{\mathcal{X}} \hat E(x)\mathrm  dx}  
\end{align*}
from equation \eqref{eq:alphastar}, where the second equality arises from the definition of $\hat \alpha$ and if we set the weighting to $w(x) = \hat E(x)$. 
%\[
%\bar \alpha^* = \frac { \int_{\mathcal{X}} \hat O(x) \mathrm dx } {\int_{\mathcal{X}} \hat E(x)\mathrm  dx}  %\overset{LC}{=}  \frac { \sum_{i=1}^n  \int \mathrm dN_i(s) } {\sum_{i=1}^n \int Y_i(s) \mathrm  ds}, % = \frac{\int \mathrm dN(s) }{\int Y(s) \mathrm ds},
%\]
%where we write denote $N(s) = \sum_{i=1}^d N_i(s)$ and $Y(s) = \sum_{i=1}^d Y_i(s)$. 
One can further reduce the estimator to 
\begin{equation} \label{eq:alphastar:simple}
\bar \alpha^{*} = \frac { \sum_{i=1}^n  \int \mathrm dN_i(s) } {\sum_{i=1}^n \int Y_i(s) \mathrm  ds}. 
\end{equation}
%\begin{center} \textbf{ also for LL(?) } \end{center}
This simplification is  due to the normalization $\int   K_h(x,X_i(s)) \mathrm d x =1$ of the kernel function $K_h$ in \eqref{eq:boundedkernel}. The estimator $\bar \alpha^*$ is the additive hazard equivalent of the intercept in nonparametric regression. 
Note that in backfitting of the regression function  $m$ in \cite{Mammen:etal:99}, the estimator for the additive constant $m_0$ of the conditional mean $m$ is given as $\tilde m_0 = \bar Y_n$. 
Our result for $\bar \alpha^*$ is the total number of occurrences divided by the average exposure time. In the case of non-filtered data, $ \int \mathrm dN_i(s)$ equals unity for every $i$ and thus $\bar \alpha^*=  \left( \frac 1 n \sum_{i=1}^n \int Y_i(s) \mathrm  ds\right)^{-1}$. This term is the natural survival analysis equivalent to what the empirical mean is in regression. %which one can get by replacing $\hat E$ with an estimator $\hat p$ of the probability density. 

The constant component $\alpha^*$ and all components $\alpha_j$ of the unknown underlying hazard  $\alpha$ are uniquely identified through 
\begin{equation}\label{eq:identification:theory}
\int \alpha_j(x_j) E_j(x_j) \mathrm dx_j = 0
\end{equation}
with $E_j(x_j)=\int E(x) \mathrm dx_{-j}$ for all $j$. This motivates the choice $w(x) = \hat E(x)$  in equation \eqref{eq:cond} and the notation $\hat E_k(x_k)$  instead of $w_k(x_k)$ for this choice of weighting from now on.

For the same data-adaptive weighting we simplify the terms in equation \eqref{eq:backfit} with some new notation. Analogously to the one-dimensional marginals, we write $\hat E_{k,j}(x_k,x_j) =  \int_{\mathcal{X}_{x_k,x_j}}  \hat E(x) \mathrm  dx_{-(k,j)}$ for $x_{-(k,j)} = (x_0,\dotsc,x_{j-1},x_{j+1},\dotsc, x_{k-1},x_{k+1},\dotsc,x_d)$ and $\mathcal X_{x_k,x_j} = \{ (x'_0,\dotsc,x'_d) \in \mathcal X : x'_k = x_k,x'_j = x_j\}$, i.e.\ we integrate over all components except for $x_j$ and $x_k$ which are fixed values. Analogously, we define the marginal occurrence estimator  $\hat O_k(x_k) = \int_{\mathcal X_{x_k}} \hat O(x) \mathrm dx_{-k}$.

In the local constant case investigated here, it can be easily shown that it holds 
\begin{align}
\hat O_k(x_k) &=  \frac 1 n \sum_{i=1}^n    \int {k_h(x_k, X_{ik}(s))}  \mathrm dN_i(s), \label{eq:Ok:LC}  \\
%\hat O_{k,j}(x_k,x_j) &=  \frac 1 n \sum_{i=1}^n    \int \frac{k_h(x_j- X_{ij}(s))}{\int_{\mathcal{X}_j} k_h(u_j-X_{ij}(s)) \mathrm du_j }\frac{k_h(x_k- X_{ik}(s))}{\int_{\mathcal{X}_k} k_h(u_k-X_{ik}(s)) \mathrm du_k}  \mathrm dN_i(s), \\
\hat E_k(x_k) &=  \frac 1 n \sum_{i=1}^n    \int {k_h(x_k, X_{ik}(s))}  Y_i(s) \mathrm ds, \label{eq:Ek:LC} \\
\hat E_{j,k}(x_j,x_k) &=  \frac 1 n \sum_{i=1}^n    \int {k_h(x_j, X_{ij}(s))}{k_h(x_k, X_{ik}(s))} Y_i(s) \mathrm ds,  \label{eq:Ejk:LC} 
\end{align}
for $j\neq k$ if each pair of covariates has a rectangular support. Thus, these estimators are indeed just one- and two-dimensional marginal estimators and can be computed efficiently for high dimensions $d>2$.  %The local linear version can be derived analogously. 
 
Now equation \eqref{eq:backfit} implies the backfitting equation
\begin{align} \label{eq:LC:backfiting:equation}
\bar \alpha_k(x_k) %&= \frac{ \int_{\mathcal{X}_{x_k}}  \hat O(x) -  [\bar \alpha^* - \sum_{j\neq k} \bar \alpha_j(x_j)]\hat E(x) \mathrm  dx_{-k}}{\int_{\mathcal{X}_{x_k}}  \hat E(x)  \mathrm dx_{-k}}  \\ &=  \frac{\hat O_k(x_k)}{\hat E_k(x_k)} -  \sum_{j\neq k}  \frac{\int_{\mathcal{X}_{x_k}}    \bar \alpha_j(x_j)\hat E(x)\mathrm   dx_{-k}}{\hat E_k(x_k)}  -   \bar \alpha^* \\ 
 &= \hat \alpha_k(x_k)  -  \sum_{j\neq k}  \int_{\mathcal{X}_{j}}   \bar \alpha_j(x_j) \frac{\hat E_{k,j}(x_k,x_j)}{\hat E_k(x_k)} \mathrm   dx_{j}  -   \bar \alpha^*,
\end{align}
for the notation $\hat \alpha_k(x_k)= {\hat O_k(x_k)}/{\hat E_k(x_k)}$. 
%Another representation of the last equation is 
%\begin{equation}  \label{eq:backfit:LC:theory}
%\bar \alpha_k(x_k) = \hat \alpha_k(x_k)  - \hat \alpha_{0,k} -  \sum_{j\neq k}  \int_{\mathcal{X}_{j}}   \bar \alpha_j(x_j) \left[\frac{\hat E_{k,j}(x_k,x_j)}{\hat E_k(x_k)} - \hat E_{j,[k+]}(x_j) \right] \mathrm   dx_{j} 
%,\end{equation}
%where we have used the terms
%\begin{align*}
%\hat  E_{j,[k+]}(x_j) &= \frac{\int \hat E_{k,j}(x_k,x_j) \mathrm dx_k }{\int \hat E_k(x_k)\mathrm dx_k},  \\
%\tilde \alpha_{0,k}& =  \frac{\int  \hat \alpha_k(x_k)  \hat E_k(x_k) \mathrm dx_k }{\int \hat E_k(x_k)\mathrm dx_k},
%\end{align*}
%to normalize each component in every step of the estimation. This constitutes a normalization factor during the computation and assures that condition \eqref{eq:cond} holds. 

%This representation is equivalent to equation (8) in \cite{Nielsen:Sperlich:05} with $\hat E_{k,j}(x_k,x_j)$ taking the part of marginal density estimators. 
%This is the equivalent to equation (4.2) in \cite{Hiabu:etal:19}. 

Using the last expression, we can get estimators for $\alpha_0,\dots,\alpha_d$ through iterative backfitting via
%\[ \bar \alpha^{[r+1]}_k(x_k) =   \tilde \alpha_k(x_k)  -  \bar \alpha^* -  \sum_{j < k}  \int_{\mathcal{X}_{j}}   \bar \alpha^{[r+1]}_j(x_j) \frac{\hat E_{k,j}(x_k,x_j)}{\hat E_k(x_k)} \mathrm   dx_{j}   - \sum_{j >  k}  \int_{\mathcal{X}_{j}}   \bar \alpha^{[r]}_j(x_j) \frac{\hat E_{k,j}(x_k,x_j)}{\hat E_k(x_k)} \mathrm   dx_{j}, \]
\begin{align}\begin{split} \label{eq:backfit:LC}
\bar m^{[r+1]}_k(x_k) &=   \hat \alpha_k(x_k)   -  \sum_{j < k}  \int  \bar \alpha^{[r+1]}_j(x_j) \frac{\hat E_{k,j}(x_k,x_j)}{\hat E_k(x_k)} \mathrm   dx_{j}    - \sum_{j >  k}  \int    \bar \alpha^{[r]}_j(x_j) \frac{\hat E_{k,j}(x_k,x_j)}{\hat E_k(x_k)}  \mathrm   dx_{j},  \\
\bar \alpha^{[r+1]}_k(x_k) &= \bar m^{[r+1]}_k(x_k) - \left( \int \hat E_k(x_k) \mathrm dx_k\right)^{-1}  \int \bar m^{[r+1]}_k(x_k) \hat E_k(x_k) \mathrm dx_k, 
 \end{split} 
\end{align}
for $k=1,\dotsc,d$ in step $r+1$. 
Recall that $\hat \alpha_k$, $k=0,\dots,d$, are the (non-additive)  estimators which were defined via $\hat \alpha_k(x_k)= {\hat O^{}_k(x_k)}/{\hat E^{}_k(x_k)}$.
We suggest to start with the initialization $\bar \alpha^{[0]}_k(x_k) = \hat \alpha_k(x_k)$, that is related to the one-dimensional local linear hazard estimator, see \cite{Nielsen:Tanggaard:01}.  However, these pilot estimators can be set to different estimators. %In the appendix, we introduce certain technical conditions the pilot estimator has to satisfy. 
The asymptotic theory we present here is illustrated for the choice $\hat \alpha_k$.  In Section \ref{sec:twostep} of the appendix, we illustrate how one can obtain the same estimator $\bar\alpha_k$ by first minimizing \eqref{eq:criterion} without an additive constraint, yielding the pilot estimator $\hat \alpha_k$ and then running an additive minimisation of $\hat \alpha_k$.  \\

The complete smooth backfitting algorithm for the local constant additive hazard estimator $\bar\alpha$ is as follows.
\begin{enumerate}
\item Compute $\hat O_k$, $\hat E_k$, and $\hat E_{j,k}$ from equations \eqref{eq:Ok:LC}--\eqref{eq:Ejk:LC} and set $\hat \alpha_k(x_k) =\allowbreak {\hat O_k(x_k)}/{\hat E_k(x_k)}$ for $k,j=0,\dots,d$. 
\item Set $r=0$ and $\bar \alpha^{[r]}_k= \hat \alpha_k $ for  $k=0,\dots,d$. 
\item For $k=0,\dots,d$, compute $\bar \alpha^{[r+1]}_k(x_k)$ via equation \eqref{eq:backfit:LC}  for all points $x_k$. 
\item If the convergence criterion
\[
\frac{\sum_{k=0}^d \int \left(\bar \alpha^{[r+1]}_k(x_k) - \bar \alpha^{[r]}_k(x_k) \right)^2 \mathrm dx_k}{\sum_{k=0}^d \int \left(\bar \alpha^{[r+1]}_k(x_k) \right)^2 \mathrm dx_k + 0.0001} < 0.0001
\]
is fulfilled, stop; otherwise set $r$ to $r+1$ and go to step 3.
\item After convergence in step $r$, set $\bar \alpha_k= \bar \alpha_k^{[r+1]}$ for $k=0,\dots,d,$ and 
% \[
% \bar \alpha(x) = \bar \alpha^{*} + \sum_{j=0}^d \bar \alpha_k(x_j), 
% \]
%for
$\bar \alpha^{*} =  { \sum_{i=1}^n  \int \mathrm dN_i(s) }/{\sum_{i=1}^n \int Y_i(s) \mathrm  ds}$.  %from equation  \eqref{eq:alphastar:simple}.
\end{enumerate}

Note that the quantities $\hat E_{j,k}(x_j,x_k)$, $\hat E_k(x_k)$, $\hat \alpha(x_k)$, and $\bar \alpha^*$  can be calculated once in the beginning and they are not updated during the iteration process. This is a computational advantage. 
However, we want to emphasize that the downside of the analogue local linear approach to this section is that the local linear pilot estimator does not necessarily exist for low numbers of observations in high dimensions. The local constant estimator on the other hand suffers from bad performance at boundaries.

\subsection{Asymptotic properties of the local constant  smooth backfitting additive kernel hazard estimator} \label{sec:result:lc}
%\subsection{Asymptotic properties} \label{sec:result:direct:projection}
We now  derive the asymptotic behavior of the local constant estimator under weak assumptions. % on the initialization and without any assumptions on the pilot estimator from Section \ref{sec:twostep}. 
Indeed, we don't assume existence of $\hat O,\hat E$ but only existence of some one- and two-dimensional marginal estimators $\hat O_k,\hat O_{k,j},\hat E_k, \hat E_{k,j}$, $j,k=0,\dots,d$, which is satisfied under the conditions illustrated below. %We want to emphasize again that the two-step local constant estimator is identical to the local constant estimator derived via direct minimisation  in equation  \eqref{eq:criterion:LC}, which is explained in  Appendix \ref{appendix:projection}%that will be introduced later in Section \ref{sec:direct}. 
%After introducing the direct estimators in the next section, we will prove the asymptotic behavior of the local linear estimator and we will need to use the representation as a direct minimisation to ensure that the local linear estimator is well defined for high dimension $d\gg 2$. 

The following conditions are sufficient to derive asymptotic normality of  % the unstructured local constant pilot estimator $\hat \alpha^{LC}(x) = \hat O^{LC}(x) / \hat E^{LC}(x)$ and 
the resulting smooth backfitting estimators $\bar \alpha_j$, $j=0,\dots,d$.  
\begin{description}
\item[A1] The exposure satisfies $\inf_{x\in \mathcal X} E(x)>0$ and its marginals  $E_j$ are differentiable for every $j$. Moreover, the conditional density $f_t$ of $Z$ given $Y(t)=1$ is continuous for every $t\in [0,T]$ and it holds  $\sup_{x\in \mathcal X} f_t(x) < C_f$ for some constant $C_f$.
\item[A2] There exists a function $\gamma \in C^2([0,\mathcal{T}])$ such that it holds $n^{-1} \sum_{i=1}^n Y_i(t) \to \gamma(t)$ in probability as $n\to\infty$ for every $t\in[0,\mathcal{T}]$. 
\item[A3] The function $k$ is a second order kernel, that is it satisfies $\int k(u)\mathrm du=1$, $\int uk(u) \mathrm du=0$. Furthermore, $k$ is a symmetric and Lipschitz continuous function with support $[-1,1]$. 
\item[A4] It holds  $n^{1/5}h\rightarrow c_h$ for a constant $0 < c_h< \infty$ as $n\rightarrow \infty$.
\item[A5] The hazard $\alpha$ is two times continuously differentiable in every component of  $x\in \mathcal X$. % and $\inf_{x\in \mathcal X} \alpha(x) >0$. 
%\begin{center} \textbf{necessary? Yes(?) } \end{center}
%\begin{center} \textbf{D6: Lipschitz continuity of $\alpha$ necessary?} \end{center}
\end{description}
Note that in our notation $\gamma(t)$ from A2 and $E_0(t)$ are almost surely identical. However, the definition of $E_0$ does not assure $E_0\in C^2([0,\mathcal{T}])$ without A2.

\begin{theorem}[Local constant smooth backfitting estimator] \label{thm:main:theorem:lc}  
Let $\hat \alpha_j = \hat O_j^{}/\hat E_j^{}$ be the pilot estimator for $j=0,\dots,d$. 
Under Assumptions A1--A5, with probability tending to 1, there exists a unique solution $\{\bar \alpha^*, \bar \alpha_j: j=0,\dots,d\}$ to \eqref{eq:criterion:LC}, and the backfitting algorithm converges to it:
$$\int\left[\bar{\alpha}_j^{[r]}\left(x_j\right)-\bar{\alpha}_j\left(x_j\right)\right]^2 E_j\left(x_j\right) \mathrm{d} x_j \rightarrow 0.$$ 
For $x_0\in(0,\mathcal{T})$ and $x_l\in(0,R)$, $l=1,\dots,d$, the solution satisfies 
\[  % components of hazard
n^{2/5}\left\{ \left(\begin{matrix} \bar \alpha_0(x_0) - \alpha_0(x_0) \\ \vdots \\ \bar \alpha_d(x_d) - \alpha_d(x_d) \end{matrix}\right) \right\} \to  \mathcal N \left( \left(\begin{matrix} c_h^2 b_0(x_0) \\ \vdots \\  c_h^2 b_d(x_d)  \end{matrix}\right), \left(\begin{matrix} v_0(x_0) & 0 & \cdots & 0 \\ 0 & \ddots && \vdots \\ \vdots &&\ddots & 0 \\ 0 & \cdots & 0 & v_d(x_d)\end{matrix}\right) \right), 
\]
and in particular $\bar \alpha(x) =\bar\alpha^* + \sum_{j=0}^d \bar \alpha_j $ with $\bar\alpha^*$ from equation \eqref{eq:alphastar:simple} 
satisfies
\[  % total hazard
n^{2/5}\left\{ \bar \alpha(x) - \alpha (x) \right\} \to  \mathcal N \left(c_h^2 \sum_{j=0}^d b_j(x_j), \sum_{j=0}^d v_j(x_j) \right), 
\]
for $n\to\infty$, where
\begin{align*}
%b_j(x_j) &=  \int u^2 k(u) \mathrm du   \left[ \alpha_j'(x_j) \frac{\partial \log E(x)}{\partial x_j }  + \frac 1 2 \alpha_j''(x_j) \right],  \\  
v_j(x_j) &= c_h^{-1}  \int k(u)^2 \mathrm du  \, \sigma_j^2(x_j) E_j(x_j)^{-1}, \\
\sigma_j^2(x_j) &=   \alpha^*  E_j(x_j)^{-1}+ \sum_{l \neq j}  \int \alpha_l(u) E_{jl}(x_j,u) E_j(x_j)^{-1} \mathrm d u + \alpha_j(x_j)  . 
\end{align*} 
and where $b_j$ is given through
\[
(b_0, b_1, \dots,b_d) = \argmin_{\mathcal B} \int \left[ \beta(x) -\beta_0 - \beta_1(x_1) - \dots - \beta_d(x_d) \right]^2 E(x) \mathrm d x,
\]
for 
\[
\beta(x) = \sum_{j=0}^d \int u^2 k(u) \mathrm du   \left[ \alpha_j'(x_j) \frac{\partial \log E(x)}{\partial x_j }  + \frac 1 2 \alpha_j''(x_j) \right],  \\  
\]
%$\beta$ 
 and $\mathcal B = \{\tilde\beta=(\beta_0,\beta_1,\dots,\beta_d) : \int \beta_j(x_j)E_j(x_j)\mathrm dx_j = 0 ; j=0,\dots,d\}$.  

%\begin{center} \textbf{double check equations for $\sigma_j, v_j$!!!} \end{center}
%\dots``we should get oracle efficiency for the local linear estimator with the bias being additive''\dots
%Note that the correction terms $\int k_h(x_j,u) \mathrm du$ only appear at the boundaries and equal unity in the interior of the support. 
\end{theorem}
The proof of Theorem \ref{thm:main:theorem:lc} is given in Appendix \ref{sec:theory:lc}.

\begin{remark}
Define the martingale
$M_i=N_i-\Lambda_i$ where $\Lambda_i$ is the compensator of $N_i$.
The term $ \int k(u)^2 \mathrm du \, \sigma_j^2(x_j) E_j(x_j)$ occurs as the asymptotic variance of the martingale $ \int k_h(x_j,X_{ij}(s))  \mathrm d M_i(s)$. 
The convergence rate is the same as for a one-dimensional local constant hazard estimator,
see e.g. \cite{Nielsen:Tanggaard:01}.
% If exposure the $Y$, survival time and all covariates are independent, it holds in particular $ \sigma_j(x_j)= \alpha^* +\alpha_j(x_j) $.
%Under the additional technical assumption $E_{jl}(x_j,x_l) =  E_j(x_j) E_l(x_l)$, one would get  $ \sigma_j(x_j)= \alpha^* +\alpha_j(x_j) $.
In the nonparametric regression setting $Y=m(X)+\varepsilon$ of \cite{Mammen:etal:99}, and in contrast to our hazard estimator, the asymptotic variance under certain regularity conditions is specified through $\sigma_j^2(x_j) = \Var(Y-m(X) | X_j = x_j)$ without any closed form expression. %It fulfills the same role as the variance of $ \int k_h(x_j,X_{ij}(s))  \mathrm d M_i(s)$ in our setting. 
%Moreover, $\beta$ is the bias of the full-dimensional estimator. 
\end{remark}
%\begin{center} \textbf{double check remark! } \end{center}

\begin{remark}\label{rem:alpha:star1}
By Lemma \ref{lem:alpha:star:parametric:rate:LC} in the appendix, $\tilde\alpha^{*}$  is an unbiased estimator of $\alpha^*$ if the identification conditions $\int \alpha_j(x_j) E_j(x_j) \mathrm dx_j=0$ hold for $j=0,\dots,d$. 
\end{remark}

\subsection{The local linear smooth backfitting additive kernel hazard estimator} \label{sec:direct}
%\begin{center}  \textbf{use indices $j,k,l$ consistently!!! }  \end{center}

%We describe the direct projection estimator for the local linear case. % only. The local constant case is analogue. 

The local linear smooth backfitting estimator $\tilde  \alpha_j(x_j)$  for $j=0,\dots, d$, can be described  by the minimisation in equation \eqref{eq:criterion:LL}. As described in Section \ref{subsec:introduce:estimators}, this is equivalent to the minimisation in \eqref{eq:criterion} for $p=1$  with respect to $(\hat \alpha, \hat \alpha^{(1)})$ under the constraints $\theta_0(x) = \hat\alpha ^* + \sum_{j=0}^d \hat\alpha_j(x_j)$, $ \theta_{1,j}(x_j)= \hat\alpha^{(1)}_j(x_j)$ for a certain  weighting function $w$. 

Denoting the estimator of derivatives $\alpha'_j$ by $\tilde \alpha^j$ in the following, the first order conditions for the minimisation in $\tilde \alpha_j(x_j) +\tilde\alpha^*$ and $\tilde \alpha^j(x_j)$ can be written as 
\begin{align}
%%\begin{split}
%\tilde \alpha_j(x_j) \hat V_{0,0}^j(x_j) +\tilde \alpha^j(x_j) \hat V_{j,0}^j(x_j) = & \frac 1 n \sum_{i=1}^n \int \int  K_h(x,X_i(s))  \mathrm  dN_i(s) \, \mathrm dx_{-j} \notag  \\
%& - \tilde \alpha^* \hat V_{0,0}^j(x_j) - \sum_{l\neq j} \int \tilde \alpha_l(x_l) \hat V_{0,0}^{l,j} (x_l,x_j) \mathrm dx_l    \label{eq:direct1} \\
%&- \sum_{l\neq j} \int \tilde \alpha^l(x_l) \hat V_{l,0}^{l,j}(x_l,x_j) \mathrm dx_l  \notag \\
[\tilde \alpha_j(x_j)+\tilde\alpha^*] \hat V_{}^j(x_j) +\tilde \alpha^j(x_j) \hat V_{j}^j(x_j) = & \frac 1 n \sum_{i=1}^n \int k_h(x_j,X_{ij}(s)) \mathrm  dN_i(s)  \notag  \\
&- \sum_{l\neq j} \int \tilde \alpha_l(x_l) \hat V_{}^{l,j} (x_l,x_j) \mathrm dx_l    \label{eq:direct1} \\
&- \sum_{l\neq j} \int \tilde \alpha^l(x_l) \hat V_{l}^{l,j}(x_l,x_j) \mathrm dx_l  \notag \\
%%\end{split}
%%\begin{split}
%\tilde \alpha_j(x_j) \hat V_{j,0}^j(x_j) +\tilde \alpha^j(x_j) \hat V_{j,j}^j(x_j) = & \frac 1 n \sum_{i=1}^n \int \int \left( \frac{x_j - X_{i,j}(s)}{h} \right) K_h(x,X_i(s))  \mathrm  dN_i(s) \, \mathrm dx_{-j},  \notag \\ 
%&- \tilde \alpha^* \hat V_{j,0}^j(x_j) - \sum_{l\neq j} \int \tilde \alpha_l(x_l) \hat V_{0,j}^{l,j} (x_l,x_j) \mathrm dx_l \label{eq:direct2}   \\
%&- \sum_{l\neq j} \int \tilde \alpha^l(x_l) \hat V_{l,j}^{l,j}(x_l,x_j) \mathrm dx_l,   \notag 
[\tilde \alpha_j(x_j) +\tilde\alpha^*]\hat V_{j}^j(x_j) +\tilde \alpha^j(x_j) \hat V_{j,j}^j(x_j) = & \frac 1 n \sum_{i=1}^n \int  \left( \frac{x_j - X_{i,j}(s)}{h} \right) k_h(x_j,X_{ij}(s))  \mathrm  dN_i(s),  \notag \\ 
& - \sum_{l\neq j} \int \tilde \alpha_l(x_l) \hat V_{j}^{l,j} (x_l,x_j) \mathrm dx_l \label{eq:direct2}   \\
&- \sum_{l\neq j} \int \tilde \alpha^l(x_l) \hat V_{l,j}^{l,j}(x_l,x_j) \mathrm dx_l,   \notag 
&%\end{split}
\end{align}
with the new  notation 
\begin{align}
\hat V_{}^j(x_j) %&= \frac 1 n \sum_{i=1}^n \int \int K_h(x,X_i(s))Y_i(s) \mathrm ds \, \mathrm dx_{-j} ,\\
&= \frac 1 n \sum_{i=1}^n \int k_h(x_j,X_{ij}(s))Y_i(s) \mathrm ds \label{eq:V:hat:first},\\
\hat V_{}^{l,j}(x_l,x_j) %&= \frac 1 n \sum_{i=1}^n \int \int K_h(x,X_i(s))Y_i(s) \mathrm ds \, \mathrm dx_{-(l,j)} , \\
&= \frac 1 n \sum_{i=1}^n \int   k_h(x_l,X_{il}(s))  k_h(x_j,X_{ij}(s))Y_i(s) \mathrm ds, \notag \\
\hat V_{j}^j(x_j) %&= \frac 1 n \sum_{i=1}^n \int \int \left( \frac{x_j - X_{i,j}(s)}{h} \right) K_h(x,X_i(s))Y_i(s) \mathrm ds \, \mathrm dx_{-j} ,\\
 &= \frac 1 n \sum_{i=1}^n \int  \left( \frac{x_j - X_{i,j}(s)}{h} \right)  k_h(x_j,X_{ij}(s))Y_i(s) \mathrm ds  ,\notag  \\
\hat V_{l}^{l,j}(x_l,x_j) %&= \frac 1 n \sum_{i=1}^n \int \int \left( \frac{x_l - X_{i,l}(s)}{h} \right) K_h(x,X_i(s))Y_i(s) \mathrm ds \, \mathrm dx_{-(l,j)} ,\\
&= \frac 1 n \sum_{i=1}^n \int  \left( \frac{x_l - X_{i,l}(s)}{h} \right)k_h(x_l,X_{il}(s))  k_h(x_j,X_{ij}(s))Y_i(s) \mathrm ds \notag ,\\
\hat V_{j}^{l,j}(x_l,x_j) %&= \frac 1 n \sum_{i=1}^n \int \int \left( \frac{x_j - X_{i,j}(s)}{h} \right) K_h(x,X_i(s))Y_i(s) \mathrm ds \, \mathrm dx_{-(l,j)},  \\
&= \frac 1 n \sum_{i=1}^n \int  \left( \frac{x_j - X_{i,j}(s)}{h} \right) k_h(x_l,X_{il}(s))  k_h(x_j,X_{ij}(s))Y_i(s) \mathrm ds , \notag  \\
\hat V_{j,j}^j(x_j) %&= \frac 1 n \sum_{i=1}^n \int \int\left( \frac{x_j - X_{i,j}(s)}{h} \right)^2  K_h(x,X_i(s))Y_i(s) \mathrm ds \, \mathrm dx_{-j},  \\
&= \frac 1 n \sum_{i=1}^n  \int \left( \frac{x_j - X_{i,j}(s)}{h} \right)^2 k_h(x_j,X_{ij}(s))Y_i(s) \mathrm ds, \notag  \\
\hat V_{l,j}^{l,j}(x_l,x_j) %&= \frac 1 n \sum_{i=1}^n \int \int\left( \frac{x_l - X_{i,l}(s)}{h} \right)\left( \frac{x_j - X_{i,j}(s)}{h} \right) K_h(x,X_i(s))Y_i(s) \mathrm ds \, \mathrm dx_{-(l,j)} \\
 &= \frac 1 n \sum_{i=1}^n \int \left( \frac{x_l - X_{i,l}(s)}{h} \right)\left( \frac{x_j - X_{i,j}(s)}{h} \right) k_h(x_l,X_{il}(s))  k_h(x_j,X_{ij}(s))Y_i(s) \mathrm ds. \label{eq:V:hat:last}
\end{align}
Here, $x_{-k}$ denotes $(x_0,\dotsc,x_{k-1},x_{k+1},\dotsc,x_d)$ and $\mathcal X_{x_k}$ denotes the set $\{ (x'_0,\dotsc,x'_d) \in \mathcal X : x'_k = x_k\}$.

Note that $\hat V_{}^j(x_j)$ and $\hat V_{}^{l,j}(x_l,x_j) $ are identical to the one- and two-dimensional local constant fits $\hat E_j(x_j)$ and $\hat E_{j,k}(x_j,x_k)$ from the local constant estimator. For simplicity of notation, we relabel them in the sequel. The terms $\hat V_{j}^j(x_j)$, $\hat V_{l}^{l,j}(x_l,x_j)$, $\hat V_{j}^{l,j}(x_l,x_j)$, $\hat V_{j,j}^j(x_j) $ and $\hat V_{l,j}^{l,j}(x_l,x_j)$ contain linear and  quadratic components, which distinguish this approach from the one in the last section. 

Furthermore, for $j=0,\dots,d$  we introduce the same identification condition as equation \eqref{eq:cond} in the local constant case and require 
\begin{equation}\label{eq:norming}
\int \tilde \alpha_j(x_j) \hat V_{}^j(x_j) \mathrm dx_j =0
\end{equation}
 to get a unique solution of \eqref{eq:direct1} and \eqref{eq:direct2}. 
%
%Due to the norming of the kernel function $K$ and the norming of $\tilde\alpha_j$ in equation \eqref{eq:norming}, the first order condition for the constant component $\alpha^*$ results in 
%\begin{equation*} %\label{eq:alphastar:LL}
% \tilde \alpha^{*} = %\frac { \sum_{i=1}^n  \int \mathrm dN_i(s) } {\sum_{i=1}^n \int Y_i(s) \mathrm  ds}  - \sum_{j=0}^d \int \tilde \alpha_j(x_j) \hat V^j_{0,0}(x_j) \mathrm dx_j - \sum_{j=0}^d \int \tilde \alpha^j(x_j) \hat V^j_{j,0}(x_j) \mathrm dx_j 
%\frac { \sum_{i=1}^n  \int \mathrm dN_i(s) } {\sum_{i=1}^n \int Y_i(s) \mathrm  ds}  - \sum_{j=0}^d \left(\int  \hat V^j_{j,0}(x_j) \mathrm dx_j \right)^{-1} \int \tilde \alpha^j(x_j) \hat V^j_{j,0}(x_j) \mathrm dx_j 
%\end{equation*}
%The term $ \tilde \alpha^{*}$ equals the estimator $ \tilde \alpha^{*}$ from the local constant case plus a term that arises because the derivative estimators $\tilde\alpha^j$ are not normalized.
%We ignore $ \tilde \alpha^{*}$ for now After convergence of the components, $\alpha^*$ is canceled off by additional bias terms in Theorem \ref{thm:main:theorem:ll}. 

 %With the additional restriction $\int \tilde \alpha^j(x_j) \hat V^j_{j,0}(x_j) \mathrm dx_j =0$, we would have $  \tilde \alpha^{*} = \bar \alpha^{*}$. 
%That second term depends on the estimates themselves and, thus, is not problematic in the asymptotic analysis. 
%
We can derive a local constant estimator from the same conditions \eqref{eq:direct1} and \eqref{eq:direct2} for $ \hat \alpha_k(x_k)$ but with $\hat\alpha'_j(x_j)$ set to zero for every $j$. If we choose $w\equiv 1$, this local constant estimator coincides with the one from Section \ref{sec:loc:const}.  \\

%  Now same argumentation as equations (32)--(43) in \cite{Mammen:etal:99} with  
% In particular can we get exactly the same equations (32)--(38), (41)--(43) with these definitions. Indeed, equations \eqref{eq:direct1} and \eqref{eq:direct2} are identical to equations (32) and (33) in \cite{Mammen:etal:99}. %,i.e. to  \[ \hat \alpha_j(x_j) = \hat \alpha_j(x_j) - .\]
%The equivalents to equations (39) and (40), respectively, are
%\begin{align}\label{eq:1dimLLfit1}
%\hat \alpha_j(x_j) \hat V_{0,0}^j(x_j) + \hat \alpha^j (x_j) \hat V_{j,0}^j (x_j) &= \frac 1 n \sum_{i=1}^n \int \int  K_h(x-X_i(s))  W(x)\mathrm  dN_i(s)\mathrm dx_{-j}, \\
%\hat \alpha_j(x_j) \hat V_{j,0}^j(x_j) + \hat \alpha^j (x_j) \hat V_{j,j}^j (x_j)  &= \frac 1 n \sum_{i=1}^n \int \int \left( \frac{x_j - X_{i,j}(s)}{h} \right) K_h(x-X_i(s))  W(x)\mathrm  dN_i(s)\mathrm %dx_{-j}, \label{eq:1dimLLfit2}
%\end{align} 
%Equations (41)--(43) from \cite{Mammen:etal:99} can be taken over identically. 

Conditions \eqref{eq:direct1}--\eqref{eq:norming} uniquely define our estimator and for the derivation of asymptotic theory \eqref{eq:direct1}--\eqref{eq:direct2}  can be written in one equation as
\begin{align}
\begin{split}
\hat M_j(x_j)  \label{eq:general:case:1}
 \left( \begin{matrix}  \tilde \alpha_j(x_j) - \hat \alpha_j(x_j)\\ \tilde \alpha^j(x_j) -\hat\alpha^j(x_j)\end{matrix} \right)
=%&\frac 1 n \sum_{i=1}^n \int \left( \begin{matrix} 1 \\ 
%h^{-1} (x_j - X_{ij}(s))\end{matrix}\right)  k_h(x_j,X_{ij}(s))   \mathrm  dN_i(s)   \\ 
-\tilde \alpha^*  \left(\begin{matrix} \hat V_{}^j (x_j) \\  \hat V_{j}^j (x_j) \end{matrix}\right)
 - \sum_{l\neq j} \int \hat S_{l,j}(x_l, x_j)  \left(\begin{matrix} \tilde \alpha_l(x_l) \\ \tilde \alpha^l(x_l) \end{matrix}\right)\mathrm dx_l, 
 \end{split}
% \int \tilde \alpha_j(x_j) \hat V_{0,0}^j (x_j) \mathrm dx_j = &0, \label{eq:general:case:2}
\end{align}
where we have used the matrices
\begin{equation}\label{eq:hatM}
\hat M_j(x_j) = \left(\begin{matrix} \hat V_{}^j (x_j)& \hat V_{j}^j (x_j) \\ 
\hat V_{j}^j (x_j) &\hat V_{j,j}^j (x_j)  
\end{matrix}\right),
\end{equation}
\begin{equation}\label{eq:hatS}
\hat S_{l,j}(x_l,x_j) = \left(\begin{matrix} \hat V_{}^{l,j} (x_l,x_j)& \hat V_{l}^{l,j}(x_l,x_j) \\ 
\hat V_{j}^{l,j} (x_l,x_j) &\hat V_{l,j}^{l,j}(x_l,x_j)
\end{matrix}\right), 
\end{equation}
and the one-dimensional local linear fit of the observations
\[
\left( \begin{matrix} \hat \alpha_j(x_j)  \\ \hat \alpha^j (x_j)   \end{matrix} \right) =   \frac 1 n \sum_{i=1}^n \int \hat M_j(x_j)^{-1} \left( \begin{matrix} 1 \\ h^{-1} (x_j - X_{ij}(s))\end{matrix}\right)  k_h(x_j,X_{ij}(s))   \mathrm  dN_i(s). 
\]
Note, that we would get the same asymptotic result for any estimator which arises from equation \eqref{eq:general:case:1} by replacing $\hat V^j_{0,0}, \hat V^j_{0,0}$ and $(\hat \alpha_j, \hat \alpha^j)$ with asymptotically equivalent estimators that satisfy the same regularity conditions in Appendix \ref{sec:theory:ll}.

%We now introduce a specific initialization $(\tilde m_0,\tilde m^{[0]}_l, \tilde m^{[0],l})$, $ l=1,\dots,d$.  We set $\tilde m_0 = 0$ and $(\tilde m^{[0]}_l, \tilde m^{[0],l})=(\hat m_l, \hat m^l)$ to the one-dimensional local linear fit of the data:
%\begin{align} \label{eq:init1}
%\hat \alpha_j(x_j) = &\left\{  (\hat V_{j,0}^j(x_j))^2 - \hat V_{j,j}^j(x_j)\hat V_{0,0}^j (x_j)\right\}^{-1}  \left[\hat V_{j,0}^j(x_j)  \hat U_j^j(x_j)  - \hat V_{j,j}^j(x_j) \hat U_0^j(x_j) \right], \\
%\hat \alpha^j (x_j) = &\left\{  (\hat V_{j,0}^j(x_j))^2 - \hat V_{j,j}^j(x_j)\hat V_{0,0}^j (x_j)\right\}^{-1}  \left[\hat V_{j,0}^j(x_j)  \hat U_0^j(x_j) - \hat V_{0,0}^j(x_j)  \hat U_j^j(x_j)  \right], \label{eq:init2}
%\end{align}
%which results as the solution for equations \eqref{eq:1dimLLfit1} and \eqref{eq:1dimLLfit2} with the definitions

For the implementation as an iterative algorithm, step $r+1$ of the backfitting algorithm is given by:
\begin{align}\label{eq:backfitalgo1}
 \left(\begin{matrix} \hat m_j(x_j)  \\ \tilde \alpha^{[r +1],j}(x_j) \end{matrix}\right)   &=  \left(\begin{matrix} \hat \alpha_j(x_j) \\ \hat \alpha^j(x_j) \end{matrix}\right) -   \hat M_j(x_j)^{-1} \sum_{l \neq j} \int \hat S_{l,j}(x_l, x_j)  \left(\begin{matrix} \tilde \alpha_l^{[r]}(x_l) \\ \tilde \alpha^{[r],l}(x_l) \end{matrix}\right) \mathrm dx_l, \\
\tilde \alpha_j^{[r+1]}(x_j) &= \hat m_j(x_j) - \left(  \int \hat V_{}^j(u_j) \mathrm du_j\right)^{-1}  \int \hat m_j(u_j) \hat V_{}^j(u_j) \mathrm du_j , \label{eq:backfitalgo2}
\end{align}
for $r=0,1,2,\dots$.

Note that $\tilde \alpha^*$ from equation \eqref{eq:general:case:1} vanishes in the component  $\alpha^{[r +1],j}(x_j) $ and it is made redundant in the other component by the norming condition \eqref{eq:backfitalgo2}.
Theorem \ref{thm:main:theorem:ll} assures the convergence of this estimator. 
%
%Regularity of the marginal components matrix $\hat M_j(x_j)$ (of dimensions $2\times 2$) for every $j=1,\dots,d$ and every $x_j$ is a much weaker property than invertibility of the matrix $D(x)$  (of dimensions $(d+1)\times (d+1)$) for every $x$ which would have been necessary for the existence of a local linear pilot estimator in Section \ref{sec:unstructured}. 
%Especially for high dimensions ($d>2$) this existence can not be assured. 

We recommend avoiding the inverse of the matrices $\hat M_j$ in the implementation for computational stability. Solving equations \eqref{eq:direct1}--\eqref{eq:direct2} for $\tilde \alpha_j(x_j)$ and $\tilde \alpha^j(x_j)$, respectively, and first replacing $\tilde \alpha^j(x_j)$ in \eqref{eq:direct1} by its latest fit  $\tilde \alpha^{[r],j}(x_j)$ and then $\tilde \alpha_j(x_j)$ in \eqref{eq:direct2} by  $\tilde \alpha_j^{[r+1]}(x_j)$ in step $r+1$, 
we get the asymptotically equivalent, more stable backfitting equations 
\begin{align} 
\begin{split}\label{eq:algo:LL1}
\tilde \alpha^{[r+1]}_j(x_j) &= \hat V^j_{}(x_j)^{-1} \Big(  \hat U^j(x_j)  -\tilde \alpha^{[r],j}(x_j) \hat V^j_{j}(x_j) - \tilde\alpha^*      \hat V^j_{}(x_j),  \\  
& - \sum_{l\neq j} \int \tilde \alpha_l^{[r]}(x_l) \hat V_{}^{l,j} (x_l,x_j) \mathrm dx_l   - \sum_{l\neq j} \int \tilde \alpha^{[r],l}(x_l) \hat V_{l}^{l,j}(x_l,x_j) \mathrm dx_l  \Big), 
\end{split}\\
\begin{split}\label{eq:algo:LL2}
\tilde \alpha^{[r+1],j}(x_j) &=  \hat V_{j,j}^j(x_j)^{-1}   \Big(  \hat U_j^j(x_j)  - \tilde\alpha_j^{[r]}(x_j)\hat V^j_{j}(x_j) - \tilde \alpha^* \hat V_{j}^j(x_j) \\
& - \sum_{l\neq j} \int \tilde \alpha^{r+1]}_l(x_l) \hat V_{j}^{l,j} (x_l,x_j) \mathrm dx_l 
- \sum_{l\neq j} \int \tilde \alpha^{[r],l}(x_l) \hat V_{l,j}^{l,j}(x_l,x_j) \mathrm dx_l \Big),  
\end{split}\end{align} 
for step $r+1$ with the notation 
\begin{align}
\hat U^j(x_j)&= \frac 1 n \sum_{i=1}^n \int  k_h(x_j,X_{ij}(s))  \mathrm  dN_i(s),\label{eq:init3} \\
%\hat U_j^j(x_j) &= \frac 1 n \sum_{i=1}^n \int \int \left( \frac{x_j - X_{ij}(s)}{h} \right) K_h(x,X_i(s))  \mathrm  dN_i(s)\mathrm dx_{-j}.\label{eq:init4}
\hat U_j^j(x_j) &= \frac 1 n \sum_{i=1}^n \int \left( \frac{x_j - X_{ij}(s)}{h} \right)  k_h(x_j,X_{ij}(s))  \mathrm  dN_i(s).\label{eq:init4}
\end{align} 
Note that $\hat U^j(x_j)$ is identical to $\hat O_j(x_j)$, the local constant occurrence estimator described in Section \ref{sec:loc:const}. 
%We set the initialization in step $r=0$ to $(\tilde \alpha^{[0]}_j(x_j), \tilde \alpha^{[0],j}(x_j)) = ( \hat\alpha_j(x_j), \hat\alpha^j(x_j))$.
We set the initialization in step $r=0$ to $(\tilde \alpha^{[0]}_j(x_j), \tilde \alpha^{[0],j}(x_j)) = (0,0)$. 
  %In matrix notation, equations \eqref{eq:init1}--\eqref{eq:init4} can be written as

The complete smooth backfitting algorithm for the local linear additive hazard estimator $\tilde \alpha$ is as follows.
\begin{enumerate}
\item Compute $\hat V_{}^j$, $\hat V_{}^{l,j}$, $\hat V_{j}^j$, $\hat V_{l}^{l,j}$, $\hat V_{j}^{l,j}$, $\hat V_{j,j}^j$, and  $\hat V_{l,j}^{l,j}$  from equations \eqref{eq:V:hat:first}--\eqref{eq:V:hat:last} and set $\hat \alpha(x_k) = {\hat O_k(x_k)}/{\hat E_k(x_k)}$ for $k,j=0,\dots,d$. 
\item Set $r=0$ and $\bar \alpha^{[r]}_k= \hat \alpha_k $ for  $k,j=0,\dots,d$. 
\item For $k=0,\dots,d$, calculate for all points $x_k$ Set $r=1$, compute $\tilde \alpha^{[r+1]}_k(x_k)$ via equations \eqref{eq:algo:LL1} and \eqref{eq:algo:LL2}. Then replace $\tilde \alpha_j^{[r+1]}$ by
\[
\tilde \alpha_j^{[r+1]}- \left(  \int \hat V_{}^j (u_j) \mathrm du_j\right)^{-1}  \int \tilde \alpha^{[r^*]}_j(u_j) \hat V_{}^j(u_j) \mathrm du_j. 
\]
\item If the convergence criterion
\[
\frac{\sum_{k=0}^d \int \left(\tilde \alpha^{[r+1]}_k(x_k) - \tilde \alpha^{[r]}_k(x_k) \right)^2 \mathrm dx_k}{\sum_{k=0}^d \int \left(\tilde \alpha^{[r+1]}_k(x_k) \right)^2 \mathrm dx_k + 0.0001} < 0.0001
\]
is fulfilled, stop; otherwise set $r$ to $r+1$ and go to step 3.
\item After convergence in step $r$, set $\tilde \alpha_k= \tilde \alpha_k^{[r+1]}$ for $k=0,\dots,d,$ and 
 $\tilde \alpha^{*} =  { \sum_{i=1}^n  \int \mathrm dN_i(s) }/{\sum_{i=1}^n \int Y_i(s) \mathrm  ds}$. %from equation  \eqref{eq:alphastar:simple}.
\end{enumerate}

%Note that equations \eqref{eq:hatM}--\eqref{eq:backfitalgo2} describe an algorithm to compute the smooth backfitting estimators of $\alpha_k$, $k=0,\dots,d$, that will converge even for different choices of $\hat V$ and different initialization $\hat \alpha_k$ under certain conditions. 

%We want to remark again, that the unstructured estimator from Section \ref{sec:unstructured} is not defined if the matrix $D$ is not invertible. 

%%Next we derive asymptotics for the smooth backfitting estimator with weak assumptions on the initialization and without any assumptions on the pilot estimator from Section \ref{sec:twostep}. Indeed, we don't assume existence of $\hat O,\hat E$ but only existence of some one- and two-dimensional marginal estimators %$\hat O_k,\hat O_{k,j},\hat E_k, \hat E_{k,j}$, $j,k=0,\dots,d$ 
% which is satisfied under rather weak conditions illustrated below. % This makes the following result more general than one we derived in Section \ref{sec:direct}. 
%%We just assume that some estimates $\hat \alpha_j$, $\hat \alpha^j$, $\hat V_j$ and $\hat V^j$, $j=1,\dots,d$ are given. 

\subsection{Asymptotic properties of the local linear smooth backfitting additive kernel hazard estimator}  \label{sec:result:direct:projection} 
For the asymptotic behavior of $\tilde\alpha_j$, we assume the same Assumptions A1--A5 as for the local constant estimator. 

\begin{theorem}[Local linear smooth backfitting estimator]  \label{thm:main:theorem:ll} 
%Let  $\hat V$ be given from equations \eqref{eq:direct1a}--\eqref{eq:direct1b}. 
%With the initialization \eqref{eq:init1}--\eqref{eq:init4} and 
Under Assumptions A1--A5, with probability tending to 1, there exists a unique solution $\{\tilde \alpha_j, \tilde \alpha^j : j=0,\dots,d\}$ to \eqref{eq:criterion:LL} and the backfitting algorithm \eqref{eq:backfitalgo1} converges to it:
\begin{align*}
&\int\left[\tilde{\alpha}_j^{[r]}\left(x_j\right)-\tilde{\alpha}_j\left(x_j\right)\right]^2 E_j\left(x_j\right) \mathrm{d} x_j \rightarrow 0, \\
&\int\left[\tilde{\alpha}^{j,[r]}\left(x_j\right)-\tilde{\alpha}^j\left(x_j\right)\right]^2 E_j\left(x_j\right) \mathrm{d} x_j \rightarrow 0.
\end{align*}

For $x_0\in(0,\mathcal{T})$ and $x_l\in(0,R)$, $l=1,\dots,d$, the solution satisfies 
\[ 
n^{2/5}\left\{ \left(\begin{matrix} \tilde \alpha_0(x_0) - \alpha_0(x_0)+ \nu_{n,0} \\ \vdots \\ \tilde \alpha_d(x_d) - \alpha_d(x_d) + \nu_{n,d}\end{matrix}\right) \right\} \to  \mathcal N \left( \left(\begin{matrix} c_h^2 b_0(x_0) \\ \vdots \\  c_h^2 b_d(x_d)  \end{matrix}\right), \left(\begin{matrix} v_0(x_0) & 0 & \cdots & 0 \\ 0 & \ddots && \vdots \\ \vdots &&\ddots & 0 \\ 0 & \cdots & 0 & v_d(x_d)\end{matrix}\right) \right), 
\]
for $n\to\infty$, where
\begin{align*}
\nu_{n,j} &= \int \int \alpha_j(x_j) k_h(x_j,u) E_j(u) \mathrm du \, \mathrm dx_j, \\ 
%b_j(x_j)  &= \frac 1 2  \alpha_j''(x_j) \int u^2 k(u) \mathrm du , \\ 
b_j(x_j)  &=\frac 1 2 \int u^2 k(u)\mathrm du \left [ \alpha_j''(x_j) - \int \alpha_j''(x_j) E_j(x_j) \mathrm dx_j\right],  \\
v_j(x_j) &= c_h^{-1} \int k(u)^2 \mathrm du  \, \sigma_j^2(x_j) E_j(x_j)^{-1}, \\
\sigma_j^2(x_j) &=   \alpha^*E_j(x_j)^{-1} + \sum_{l \neq j}  \int \alpha_l(u) E_{jl}(x_j,u) E_j(x_j)^{-1} \mathrm d u + \alpha_j(x_j)  . 
\end{align*}
This result yields in particular 
\[  % total hazard
n^{2/5}\left\{ \tilde \alpha(x) - \alpha (x) \right\} \to  \mathcal N \left(c_h^2 \sum_{j=0}^d b_j(x_j), \sum_{j=0}^d v_j(x_j) \right), 
\]
for $\tilde \alpha(x) =\tilde \alpha^* + \sum_{j=0}^d \tilde \alpha_j(x_j)$ with $ \tilde\alpha^{*} =  { \sum_{i=1}^n  \int \mathrm dN_i(s) }/{\sum_{i=1}^n \int Y_i(s) \mathrm  ds}$. 
%\dots``we should get oracle efficiency for the local linear estimator with the bias being additive''\dots
\end{theorem}

The proof of Theorem \ref{thm:main:theorem:ll} is given in Appendix \ref{sec:theory:ll}.

%\begin{center} \textbf{delete/change the following!} \end{center}
\begin{remark}
Note that the convergence rate is the same as for a one-dimensional local linear hazard estimator,
see e.g. \cite{Nielsen:Tanggaard:01}.
Furthermore, $\tilde\alpha_j(x_j)$ estimates $\alpha_j(x_j) - \int \alpha_j(x_j) \hat V^j_{}(x_j)\mathrm dx_j$ instead of $\alpha_j(x_j)$. The terms $\nu_{n,j}$ correct for this shift in the estimation of each component. 
The sum $\sum_{j=0}^d \nu_{n,j}$ vanishes as the additive adjustments cancel each other off. 

%Furthermore, it holds 
%\[
%\alpha^* = \sum_{j=0}^d \int \alpha_j(x_j) E_j(x_j) \mathrm d x_j 
%\]
%because of the identification condition in equation \eqref{eq:identification:theory}. 
%Hence 
%\[
% \sum_{j=0}^d \nu_{n,j} - \alpha^*=  \sum_{j=0}^d   \int \int \alpha_j(x_j) k_h(x_j,u) E_j(u) \mathrm du \, \mathrm dx_j -  \sum_{j=0}^d \int \alpha_j(x_j) E_j(x_j) \mathrm d x_j  
%\]
%vanishes asymptotically. 
%\[
%\nu(x_j) =\frac 12  h^2 \int u^2 k(u) \mathrm d u \int \alpha_j(x_j) \mathrm d x_j = O(h^2),
%\]
%which can be seen through a Taylor expansion of $E_j$ around $x_j$. They are just given for completeness in the first asymptotic statement in Theorem \ref{thm:main:theorem:ll} .
The component $\tilde\alpha^*$ of the estimator $\tilde \alpha$, which estimates $\alpha^*$, is identical to $\bar\alpha^*$ from the local constant case. Its asymptotic behavior is explained in Remark  \ref{rem:alpha:star1}.
\end{remark}

\section{Simulation Study}
%In this section we conduct a simulation study in order to assess finite sample properties of the proposed estimators.
\subsection{Simulation Setting}

We assume that the survival times $T_i$  follows a Gompertz–Makeham distribution, with hazard function given by
\[ 
\alpha(t,Z_i)=\alpha_0(t)+\sum_{k=1}^d \alpha_k(Z_{ik})=e^{0.01 t}+ \frac{4}{\sqrt d}\sum_{k=1}^d(-1)^{k+1} \mathrm{sin}(\pi Z_{ik}),
\]
$(i=1,\dots,n)$.
We add right censoring  with censoring variables $C_i$ that follows the same distribution as $T_i$, except the scale parameter being divided by 1.75.
The factor $4d^{-1/2}$ is chosen so that the distribution of $T_i$ doesn't vary too much in the number of covariates $d$.
Note that for convenience the components are differently identified than 
in \eqref{eq:norming}, \eqref{eq:identification:theory}.
%
%
%We used the following two models:
%\begin{align*}
%&\text{Model 1: } \eta_k(z_k)=\begin{cases}-z_k \quad &\text{if k is odd,}\\ 2z_k \quad &\text{if k is even.}\end{cases}\\
%&\text{Model 2: } \eta_k(z_k)=\begin{cases}2\sin(\pi z_k) \quad &\text{if k is odd,}\\ 2z_k \quad &\text{if k is even.}\end{cases}
%\end{align*}
We now describe how the  covariates $(Z_{i1}, \dots, Z_{id})$ are generated.
We first  simulate $(\widetilde Z_{i1}, \dots, \widetilde Z_{id})$ from a $d$-dimensional  multi-normal distribution with mean equal
0 and $\textrm{Corr}(Z_{ij},Z_{il})=\rho$ if $j\neq l,$ else 1.
Afterwards we set
\[
        Z_{ik}=2.5\pi^{-1}\text{arctan}(\widetilde Z_{ik}).
\]
We repeat the procedure and take the first $i=1,\dots,n$ observations such  that $4d^{-1/2} \sum_{k=1}^d(-1)^{k+1} \mathrm{sin}(\pi Z_{ik})$ is positive.
Technically, the values of the covariates are conditioned such that the resulting hazard is positive, and hence well defined.

%We will consider combinations of the following parameter settings:
%\begin{align*}
%n= 200, 500,  \quad d=2,99  \quad \rho=0,0.5,0.8.
%\end{align*}
%After trying several bandwidths,  if not said otherwise, we present the results for a bandwidth $b=0.3$ for both estimators in every model.

 As kernel function k, we used the Epanechnikov kernel. Performance is measured via the integrated squared error:

\begin{align*}
MISE_k=n^{-1}\sum_{i}\left(\eta_{k}(Z_{ik})-\widehat \eta_{k}(Z_{i_k})\right)^2.
\end{align*}
\subsection{Simulation Results}
We compare the performance of the local linear smooth backfitting estimator to the local constant smooth backfitting estimator.
We also compare these proposed estimators to a version of the classical backfitting equivalent where only the updated component is smoothed, see \cite{buja1989linear}.

Figure \ref{sim_figure} shows the estimation results for the first component  from 100 simulations in a setting with sample size $n=5000$, dimension $d=3$, and correlation $\rho=0.5$, calculated with a MISE optimal bandwidth.
We find that the classical backfitting estimators produce in more noise than their smooth backfitting counterpart.  The local constant smooth backfitting estimator is less smooth (more "wiggly") than the local linear version. 
This first impression can be further verified in Table \ref{sim_table}: Classical backfitting estimators perform significantly worse than the smooth alternatives.
The local linear classical backfitting estimator only gives sensible results in the easiest settings, that is when $n=5000$ and or $d=3$, while breaking down in all other cases.
  Another observation is that that the local linear smooth backfitting estimator is nearly always to be preferred to local constant smooth backfitting estimator.
  Only in the most challenging setting, i.e., $n=500$, $d=30$, did the local constant smooth backfitting estimator outperform the local linear version.
  But even in that case the advantage is only by a small margin.

\begin{figure}[h]
\includegraphics[width=15cm]{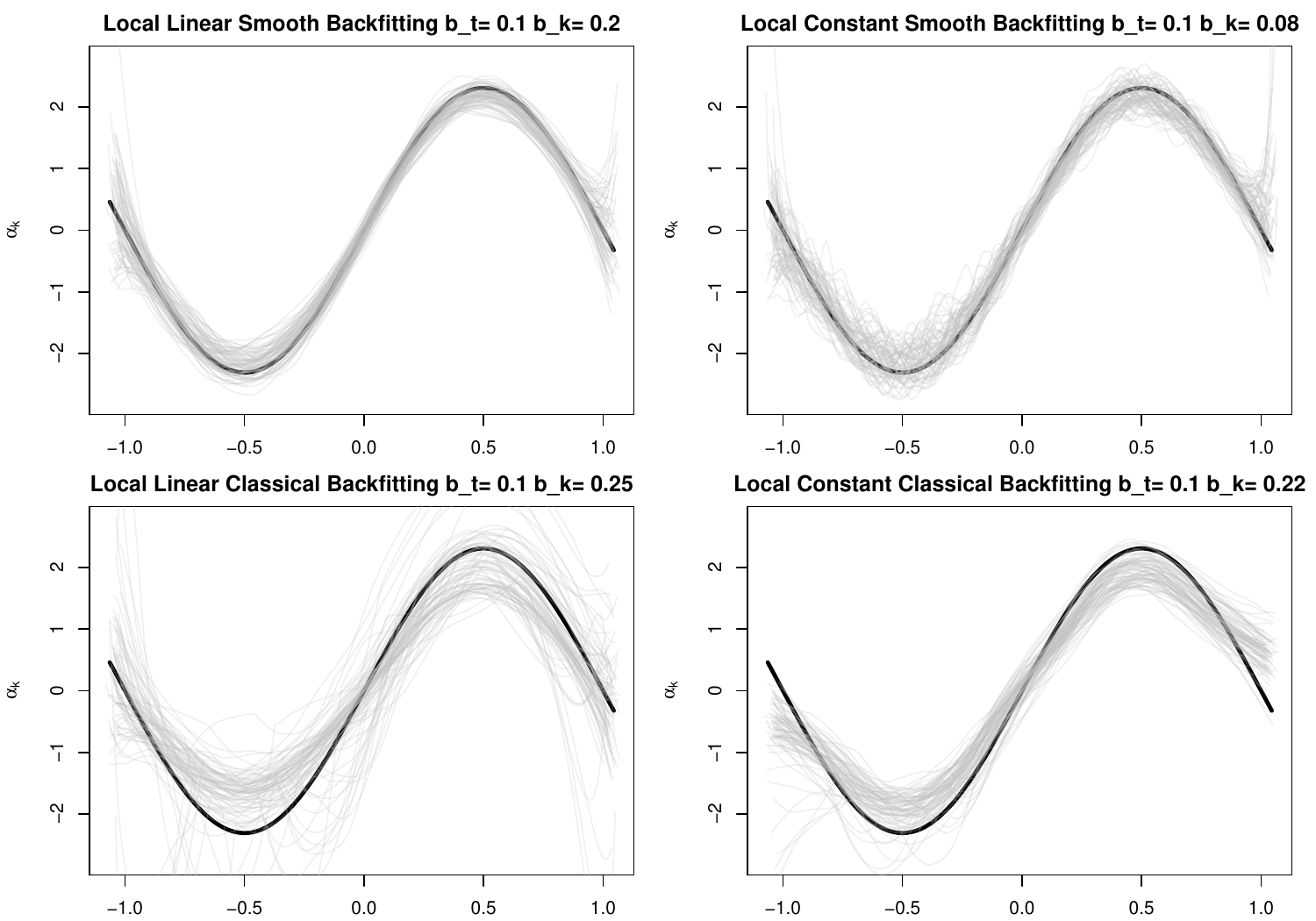}
\caption{Simulation results for $k=1$ comparing four different estimators: local constant smooth backfitting, local linear smooth backfitting, local constant backfitting, local linear backfitting. The grey lines are  represent 100 Monte Carlo simulations with MISE optimal bandwidth estimating the true curve (black).}
\label{sim_figure}
\end{figure}

\begin{table}
\begin{center}

\begin{tabular}{ |c| c c c | c c c |}
\hline
\multicolumn{7}{|c|}{d=3} \\
\hline
 &\multicolumn{3}{|c|}{n=500} & \multicolumn{3}{|c|}{n=5000} \\
&MISE& $\mathrm{Bias^2}$ & Variance &MISE& $\mathrm{Bias^2}$ & Variance  \\ 
\hline
LL-SBF & 0.25& 0.07& 0.17& 0.031& 0.007& 0.024\\  
 LC-SBF &0.30& 0.05&  0.25& 0.051& 0.011& 0.041  \\
LL-BF & 43.14& 0.69& 42.46 & 0.779& 0.041& 0.737\\  
LC-BF &1.44& 0.48& 0.96  & 0.077 &0.020&0.058 \\
\hline 
\multicolumn{7}{|c|}{d=10} \\
\hline
 &\multicolumn{3}{|c|}{n=500} & \multicolumn{3}{|c|}{n=5000} \\
&MISE& $\mathrm{Bias^2}$ & Variance &MISE& $\mathrm{Bias^2}$ & Variance  \\ 
\hline
LL-SBF &  0.22& 0.05& 0.17 &0.020& 0.005& 0.015\\  
 LC-SBF &0.24& 0.08 &0.17 & 0.030&0.006& 0.025\\
LL-BF & 1118.80&10.88& 1107.91&0.135&  0.057& 0.078\\  
LC-BF &1.02& 0.03& 0.99 & 0.031&  0.005&0.026\\
\hline 
\multicolumn{7}{|c|}{d=30} \\
\hline
 &\multicolumn{3}{|c|}{n=500} & \multicolumn{3}{|c|}{n=5000} \\
&MISE& $\mathrm{Bias^2}$ & Variance &MISE& $\mathrm{Bias^2}$ & Variance  \\ 
\hline
LL-SBF &  0.18& 0.03 & 0.15&0.014&0.0007&  0.0133\\  
 LC-SBF &0.16& 0.05&  0.10&0.029&  0.0172&  0.0114  \\
LL-BF &  NA &  NA &  NA & 0.171&0.1494&  0.0217\\
LC-BF & NA &  NA&  NA& 0.033& 0.0227& 0.0105  \\
\hline
\end{tabular}
\end{center}
\caption{Simulation results comparing four different estimators: local constant smooth backfitting, local linear smooth backfitting, local constant backfitting, local linear backfitting. Values are calculated from 500 Monte Carlo simulations with MISE optimal bandwidth.}
\label{sim_table}
\end{table}

\section{Data application: The TRACE study} \label{sec:tracestudy}

The TRACE study group (see e.g. \cite{jensen1997does}) has
collected information on more than 4000 consecutive patients with
acute myocardial infarction (AMI) with the aim of 
studying the prognostic importance of 
ventricular fibrillation (vf) on mortality.
We here consider a subset of these
patients that are available in the \texttt{timereg} R package. 
We furthermore only consider those patients with
more than 40 years of age, and only consider the first five years of follow-up time
after the diagnosis. This results in $n=1799$ observations.
At entry, i.e., time of AMI occurrence, the
patients had various risk factors recorded.
Here, additionally to duration, i.e. time since AMI occcurence, we will consider age at AMI occcurence of the patient, $a_i,$ and wall motion index (heart pumping effect based on ultrasound measurements where 2 is normal and 0 is worst \citep{scheike2009timereg}),
$\textrm{wmi}_i$.
We will ignore additional binary covariates that have been recorded as our framework only covers continuous covariates. With that regard, this section should be seen as a simple illustration of our theoretical work rather than a serious attempt to answer a real-world question.
In summary, we consider the model
\[
\lambda_i(t)=Y_i(t)\{\alpha_0(t)+\alpha_2(a_i)+ \alpha_3(\textrm{wmi}_i)\},
\]
under the identifiability condition $\int \alpha_j(x_j) d{x_j}=0$ for $j=1,2$. 
The initially estimated curve for $\alpha_0$ can be seen in Figure \ref{fig:full}.
\begin{figure}
\include{full.tex}
\caption{
Local linear additive smooth backfitting fit  of $\alpha_0$ on the full data.}
\label{fig:full}
\end{figure}
We find that the duration effect has two distinct periods with an increased risk in the beginning that flattens after approximately three months.
This suggests that it might be beneficial to apply two different amounts of smoothing on those two periods.  We therefore generate two different data sets from our original data set:
The first data set covers the risk in the first three months (this can be achieved by censoring all patients who survived beyond three months) and the second data set covers the risk conditional on surviving the first three months (i.e., omits all patients in the data set with failure or censoring in the first three months). 
The results with our local linear estimator for the two different cohorts, i.e.,  those with
ventricular fibrillation (vf=1) and those without ventricular fibrillation (vf=0) 
can is depicted in Figures \ref{fig:app1} and \ref{fig:app2}.
\begin{figure}
\include{early_add.tex}
\caption{
 Local linear fit of $(\alpha_0,\alpha_1,\alpha_2)$ for the first three months
 for two different strata depending on the value of vf. Dashed line indicates asymptotic 95\% point-wise confidence interval.
}
\label{fig:app1}
\end{figure}

 \begin{figure}
 \include{late_add.tex}
 \caption{
 Local linear fit of $(\alpha_1,\alpha_2,\alpha_3)$ conditional on surviving the first three months
 for two different strata depending on the value of vf. Dashed line indicates asymptotic 95\% point-wise confidence interval.
 }
 \label{fig:app2}
 \end{figure}
% \begin{figure}
% \include{figure1.2.tex}
 %\caption{
 % Local linear fit of $(\alpha_1,\alpha_2,\alpha_3)$
 %for two different strata depending on the value of vf. Dashed line indicates asymptotic point-wise confidence %interval.
 %}
 %\label{fig:app2}
 %\end{figure}

%\begin{figure}
%\include{figure1.3.tex}
%\caption{
%Local constant multiplicative smooth backfitting fit of of $(\alpha_1,\alpha_2,\alpha_3)$
%for four different strata depending on the value of the two binary variables (vf, chf). Dashed line indicates 
%}
%\label{fig:app:mult}
%\end{figure}

The smoothing parameter was chosen manually:
For the cohort with $\textrm{vf}=0$ we have $n=1655$ patients when considering the first three months
and chose the bandwidths for $(t,a,\text{wmi})$ as $(0.1,15,0.8)$; for the data set after surviving the first three months we have
% with an empirical  censoring  probability of $89\%$
$n=1482$ and chose a bandwidth of $(1,15,0.8)$. %with a censoring probability of $68\%$.
For the cohort with $\text{vf}=1$ we have $n=132$ patients for the first three months
and chose a bandwidth of  $(t,a,\text{wmi})$ as $(0.1,20,0.8)$;
%with an empirical  censoring  probability of $56\%$ or
for the data set after surviving the first three months we have
$n=75$ and chose a bandwidth of $(t,a,\text{wmi})$ as $(1,20,0.8)$. %with a censoring probability of $68\%$.
 The dashed lines show a point-wise asymptotic 95\% confidence interval based on Theorem \ref{thm:main:theorem:ll}.
Note that it is hereby in particular assumed that (a) the bias can be neglected and (b) that the true underlying model is indeed additive.
Therefore, the confidence intervals should be seen as rather optimistic. They nevertheless give an impression of the uncertainty under optimal conditions.
Looking at Figure \ref{fig:app1}, we find that in the first three months $\text{vf}=1$ leads to a significant increase in mortality risk.
We also find that the risk increase is more severe for older patients.
Figure \ref{fig:app2} does not provide evidence that $\text{vf}=1$ leads to an increased risk after surviving the first three months.
In the next section, we want to look at how confident we can be with the model results.

\subsection{Model robustness}

\subsubsection{CRPS score} \label{sec:cprs}
We transform our estimated hazard function $\alpha= \alpha_0+\alpha_1+\alpha_2$ into a plug-in estimator of the  survival function via the relationship $S(t|z)= \prod_{s\leq t} (1-\alpha(s,z) \mathrm ds)$. 
We then split our data randomly into an $80\%$ training set and $20\%$ test set.  We train our model
on the training set and evaluate the CPRS score  \citep{avati2020countdown} on the test set (note that a lower score indicates better performance):
\[
CRPS= m^{-1} \sum_{i=1}^m \int_0^{T_i} (1-\hat S(s|z_i))^2 \mathrm ds + \delta_i \int_{T_i}^{\infty} \hat S(s|z_i)^2 \mathrm ds,
\]
where $m$ is the size of the test set. 
Due to the additive structure, our survival prediction -- although consistent -- can still be negative. We therefore consider a simple adjustment where we numerically calculate
\[
\hat S^{\text{adj}}(s|z)= \prod_{s\leq t} (1-\hat \alpha^{adj}(s,z) \mathrm ds), \quad 
\hat \alpha^{adj}(s,z)=\max(\hat \alpha(s,z), 0).
\]
Lastly, we compare our local linear additive fit with the local constant multiplicative smooth backfitting estimator from \cite{hiabu2021smooth}. The results from 200 simulation runs can be seen in Figure \ref{fig:boxplots}.
\begin{figure}
\include{boxplots.tex}
\caption{
CPRS scores from 200 simulations of a 80/20 training-test-split. Boxplots are given for the four different data sets as
described on top of the plots and each time for three different models: smooth backfitting additive model, smooth backfitting additive model using the adjusted survival estimates $\hat S^{\text{adj}}(s)$ and the smooth backfitting multiplicative model from \cite{hiabu2021nonsmooth}.}
\label{fig:boxplots}
\end{figure}
% \begin{table}
% \centering
% \begin{tabular}{l|lll|llll}
%   &      & vf=0    &      &  &      & vf=1    &      \\ \hline
%   & Sex  & Diabetes & chf  &  & Sex  & Diabetes & chf  \\
% 0 & 29\% & 90\%     & 51\% &  & 36\% & 91\%     & 33\% \\
% 1 & 71\% & 10\%     & 49\% &  & 64\% & 9\%      & 67\% 
% \end{tabular}
% \caption{Prevalence of the risk-factors sex, diabetes, chf in our data set for the cohorts with vhf=0 and vhf=1. }
% \label{tab:vhf}
% \end{table}
We have two main observations. Firstly, the model choice does not seem to have a big impact when considering survival conditional on surviving the first three months. Secondly, for survival during the first three months using the adjusted survival probability estimates improves the performance but even better performance can be achieved by using a multiplicative model.
Nevertheless, we want to emphasize that our smooth backfitting additive estimators has the desirable projection property that if the additive  model assumption is violated the estimators converge to the closest additive fit, making the results therefore still interpretable.
We investigate this property in the next subsection.

\subsubsection{Stability under model misspecification}
We take the estimated multiplicative smooth backfitting model from the previous subsection, see also
Figures \ref{fig:app:mult1} and \ref{fig:app:mult2} in the Appendix,
as true model and investigate how 
in this case our additive estimator would look. When generating the four data sets ($\text{vf}=0,1$; risk in the first three months, risk conditional on surviving the first three months), we keep the same number of samples as in the original data sets while sampling (a, wmi) with replacement from the original data sets. Afterwards, for each row, we draw  a survival time from the
 multiplicative smooth backfitting model. The survival time is considered censored if it is greater than 0.25 when considering the first three months, and it is considered censored if it is greater than 5 when considering the period after the first three months.

We compare our additive smooth backfitting estimator to a somewhat optimal fit. Note that it is not clear how to derive an optimal fit analytically or even numerically  as it depends on the joint distribution of duration, age and wmi; which is not known.
Therefore, we approximate the optimal fit by estimating an additive smooth backfitting regression function
\citep{mammen1999existence, 10.1007/978-3-031-30114-8_5} based on 10,000 observations where the response is the known hazard. 
We consider 200 simulations and the fact that the regression estimator does not vary much as a good indicator giving us confidence that it is a good approximation of the optimal additive fit.
The results are given in Figures \ref{fig:sim1} and \ref{fig:sim2}.
We find that our proposed estimators (grey lines) -- despite the limited sample sizes -- are reasonably close to the regression fit such that we can conclude that our approach is working reasonably well in estimating
the optimal additive fit.
Lastly, it should be noted that we also tried a classical backfitting approach with kernel smoothers with the result that the estimators for all components diverged in every simulation run and did not provide any result.

\begin{figure}
\includegraphics[width=0.9\textwidth]{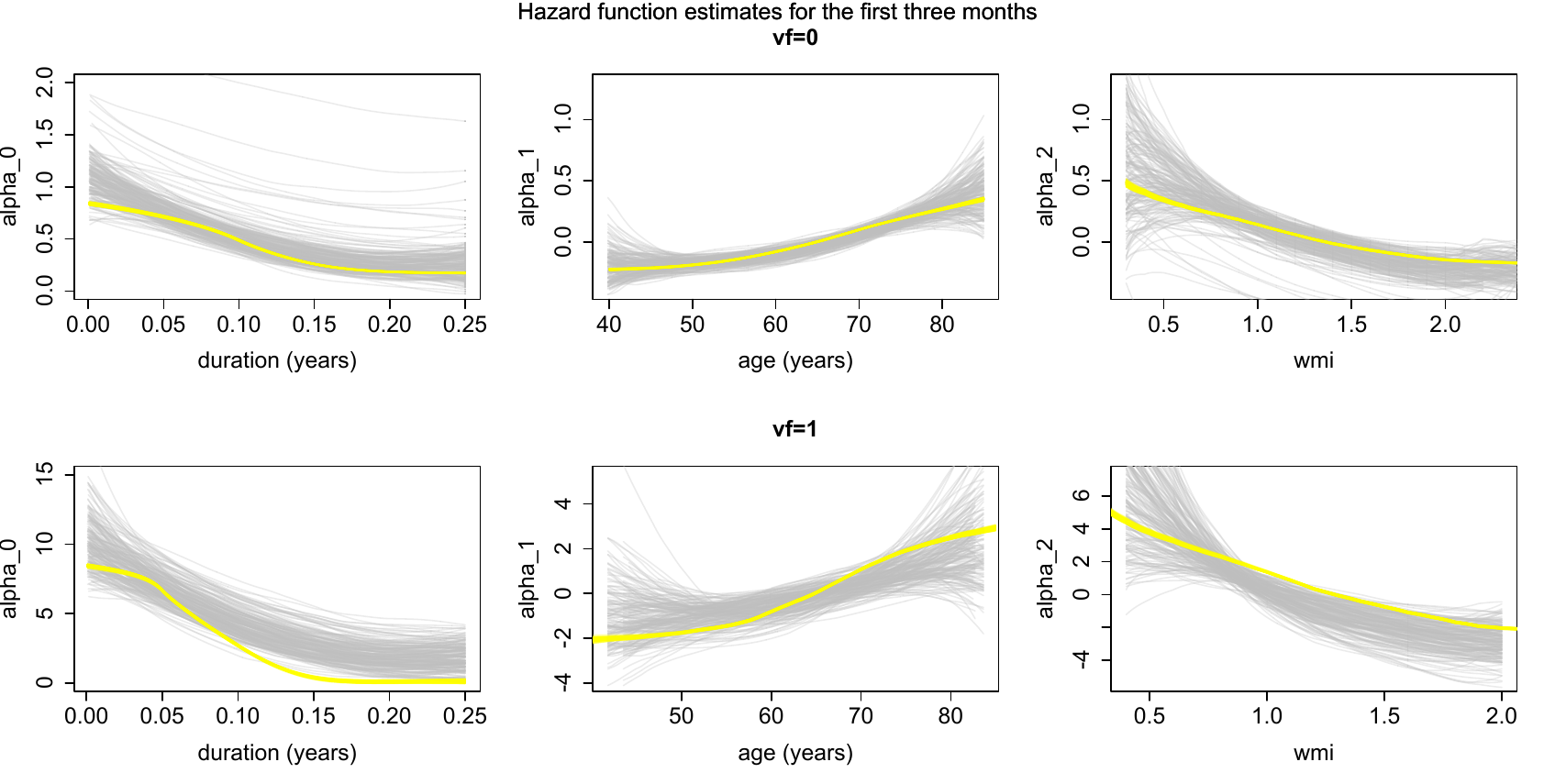}
\caption{200 simulations from a multiplicative hazard model, see Figure \ref{fig:app:mult1}. Grey curves are fitted local linear smooth backfitting estimators. Yellow curves are approximately optimal additive fits derived from a smooth backfitting additive regression fit with the true hazard as response and an inflated sample size of 10,000.}
\label{fig:sim1}
\end{figure}

\begin{figure}
\includegraphics[width=0.9\textwidth]{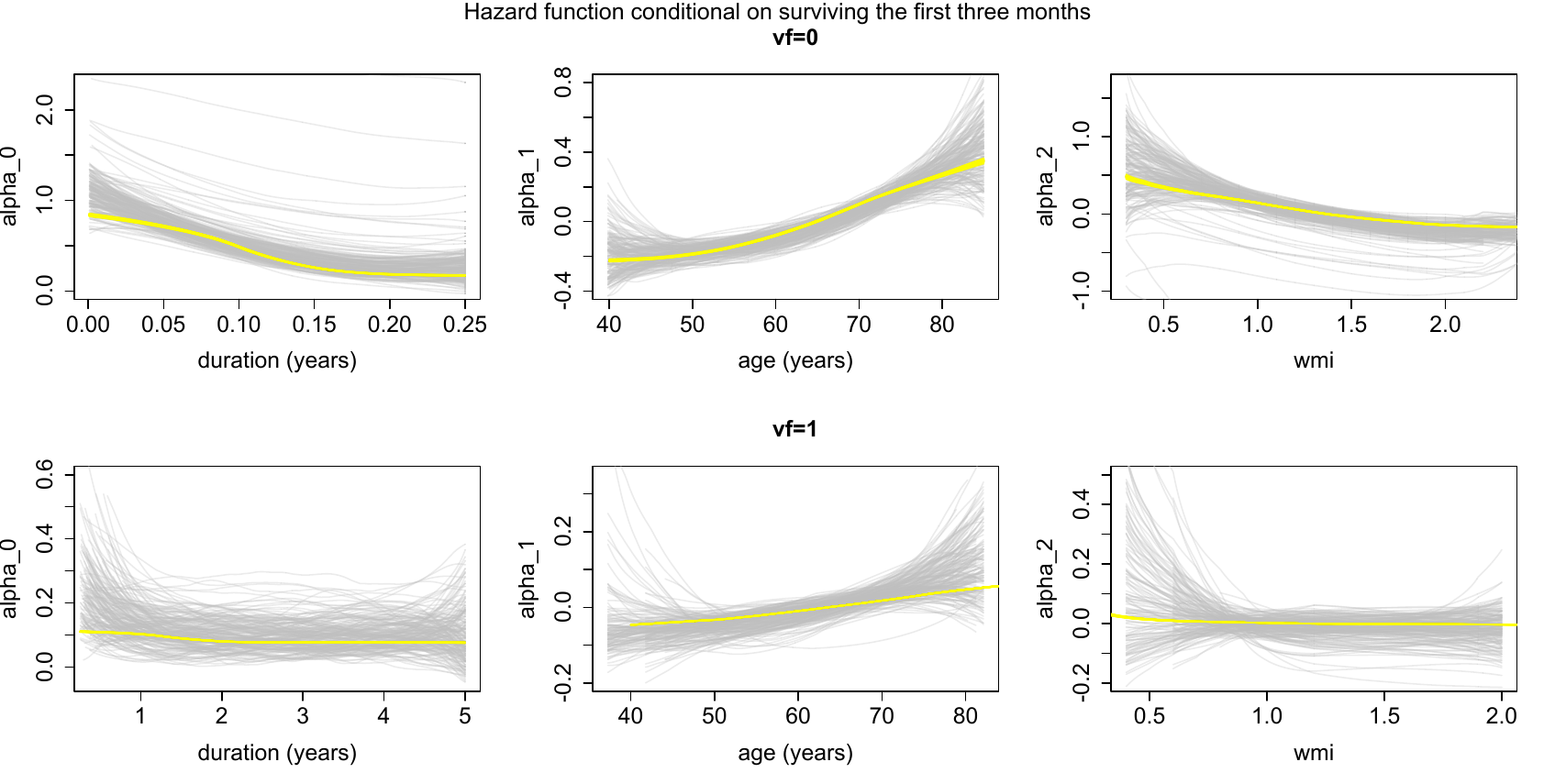}
\caption{200 simulations from a multiplicative hazard model, see Figure \ref{fig:app:mult2}.  Grey curves are fitted local linear additive smooth backfitting estimates. Yellow curves are approximately optimal additive fits derived from a smooth backfitting additive regression estimator with the true hazard as response and an inflated sample size of 10,000.}
\label{fig:sim2}
\end{figure}

\appendix

%\section{Appendix: theory for direct projection} \label{sec:theory:direct}
\section{Appendix} 
\subsection{Asymptotic theory for the local constant estimator} \label{sec:theory:lc}
For the proof of Theorem \ref{thm:main:theorem:lc}, we apply the general theory for smooth backfitting estimators. 
We split the estimator into a stochastic part and a part consisting of its bias plus a function that vanishes. For counting processes martingales, these two parts are usually referred to as the variable and the stable part, respectively. One has to show three things: the convergence of the backfitting algorithm, asymptotic normality of the stochastic part and that the bias part vanishes asymptotically. 
In \cite{Mammen:etal:99},  conditions for these three properties have been stated for a nonparametric regression setup. 
The main part of our proof is to verify these conditions under  Assumptions A1--A5. 
For completeness we restate the modified conditions in our notation. 

We also state propositions from \cite{Mammen:etal:99}, adapted to our notation, which imply the properties we need if the following assumptions hold. The difference to \cite{Mammen:etal:99} is that we make use of martingale properties and counting process theory instead of the usual arguments for kernel density estimators. 

We start with assumptions about the marginal exposures and convergence of marginal exposure estimators. Note that we don't assume any particular definition of $\hat E_j$ and $\hat E_{j,k}$, $j,k=0,\dots, d$, for the following propositions. 
\begin{description}
\item[B1] For all $ j\neq k$ it holds 
\[   \int \frac{E_{j,k}(x_j,x_k)^2}{E_j(x_j)E_k(x_k)} \mathrm d x_j \, \mathrm d x_k < \infty.   \]
\item[B2] It holds 
\begin{align*}
\int \left[\frac{\hat E_j(x_j) - E_j(x_j)}{E_j(x_j)}\right]^2 E_j(x_j)\mathrm dx_j &= o_P(1), \\
\int \left[\frac{\hat E_{j,k}(x_j,x_k)}{E_j(x_j)E_k(x_k)}  - \frac{E_{j,k}(x_j,x_k)}{E_j(x_j)E_k(x_k)}\right]^2 E_j(x_j)E_k(x_k)\mathrm dx_j \, \mathrm dx_k & = o_P(1), \\
\int \left[\frac{\hat E_{j,k}(x_j,x_k)}{\hat E_j(x_j)E_k(x_k)}  - \frac{E_{j,k}(x_j,x_k)}{E_j(x_j)E_k(x_k)}\right]^2 E_j(x_j)E_k(x_k)\mathrm dx_j \, \mathrm dx_k & = o_P(1). \\
\end{align*}
Moreover, $\hat E_j$ vanishes outside the support of $E_j$, $\hat E_{j,k}$ vanishes outside the support of $E_{j,k}$ and $\hat E$ is symmetric, i.e. $\hat E_{j,k}(x_j,x_k) = \hat E_ {k,j}(x_k, x_j)$. 
\end{description}
We assume that the marginal pilot estimator and proportions of the marginal exposure estimators are somehow bounded in probability:
\begin{description}
\item[B3] There exists a constant $C$ such that with probability tending to 1 for all $j$, 
\[
\int \hat \alpha_j(x_j)^2E_j(x_j) \mathrm dx_j \leq C. 
\]
\item[B4]  For some finite intervals $S_j\subset \mathbb R$ that are contained in the support of $E_j$, $j=1,\dots,d$, we suppose that there exists a finite constant $C$ such that with probability tending to 1 for all $j\neq k$,
\[
\sup_{x_j\in S_j} \int\frac{\hat E_{j,k}(x_j,x_k)}{ E_j(x_j)\hat E_k(x_k)^2} \mathrm dx_k \leq C. 
\]
\end{description}
We now introduce the notation $\hat \alpha_j = \hat \alpha_j^A +\hat \alpha_j^B$ for the one-dimensional smoother with
\[
\hat \alpha_j^A = {\hat E_j(x_j)}^{-1}{  \frac 1 n \sum_{i=1}^n \int k_h(x_j,X_{ij}(s)) \mathrm d M_i(s)}, 
\]
the variable part and 
\[
\hat \alpha_j^B=   {\hat E_j(x_j)}^{-1}{  \frac 1 n \sum_{i=1}^n \int k_h(x_j,X_{ij}(s)) \mathrm d \Lambda_i(s)}, \]
the stable part of $\hat \alpha_j$. Here, the compensator $\Lambda_i$ of $N_i$ is defined such that $M_i$ is a martingale and $N_i= M_i + \Lambda_i$. The definition of $M_i$ will be given later. Now we define the stochastic and stable components of the local constant smooth backfitting estimator,  $\bar \alpha_{0,j}^s$, $\bar \alpha_j^s$, for $s\in \{A,B\}$, as the solution of 
\begin{equation}  \label{eq:backfit:s:LC} 
\bar \alpha_k^s(x_k) = \hat \alpha_k^s(x_k)  - \hat \alpha^s_{0,k} -  \sum_{j\neq k}  \int_{\mathcal{X}_{j}}   \bar \alpha^s_j(x_j) \left[\frac{\hat E_{j,k}(x_j,x_k)}{\hat E_k(x_k)} - \hat E_{j,[k+]}(x_j) \right] \mathrm   dx_{j} 
,\end{equation}
where $\hat \alpha^s_{0,k} =  {\int  \hat \alpha^s_k(x_k)  \hat E_k(x_k) \mathrm dx_k } / {\int \hat E_k(x_k)\mathrm dx_k}$. 
Existence and uniqueness of $\hat \alpha_k^A, \hat \alpha^B_k$  is stated in Proposition \ref{thm:conv:of:backfitting:lc}  under the following assumptions. Assumption B6 assures converges of the variable part whereas B7 will be used for the structure of the bias part. 
\begin{description}
\item[B5] There exists a constant $C$ such that with probability tending to 1 for all $j$, it holds
\begin{align*}
\int \hat \alpha_j^A(x_j)^2 E_j(x_j) \mathrm dx_j \leq C , \\
\int \hat \alpha_j^B(x_j)^2 E_j(x_j) \mathrm dx_j \leq C .
\end{align*}
\item[B6] We assume that there is a sequence $\Delta_n \to 0$ such that 
\begin{align*}
\sup_{x_k\in S_k} \left\lvert  \int \frac{\hat E_{j,k}(x_j,x_k)}{\hat E_k(x_k)}\hat \alpha^A_j(x_j) \mathrm dx_j   \right\rvert = o_P(\Delta_n), \\
\left\lVert   \int \frac{\hat E_{j,k}(x_j,x_k)}{\hat E_k(x_k)}\hat \alpha^A_j(x_j) \mathrm dx_j   \right\rVert_{2,k} = o_P(\Delta_n),
\end{align*}
where $\lVert \cdot \rVert_{2,k}$ denotes norm defined via $\lVert g \rVert_{2,k} = \int g(u)^2 E_k(u) \mathrm du$.  The sets $S_k$ have been introduced in Assumption B4.  

\item[B7] There exist deterministic functions $\mu_{n,j}$ such that 
\[
\sup_{x_j \in S_j} \left\lvert \bar \alpha_j^B(x_j) - \mu_{n,j}(x_j) \right \rvert = o_p (\Delta_n),
\]
where  $S_k$ has been introduced in Assumption B4. 
\end{description}

The following two propositions are results from \cite{Mammen:etal:99}, adapted to our setting and notation. 
Under Assumptions B1--B3 and B5, Proposition \ref{thm:conv:of:backfitting:lc} ensures that the backfitting algorithm converges and Propositions \ref{thm:stochastic:part:lc} and \ref{thm:bias:part:lc} give the asymptotic behavior of the backfitting estimator under Assumptions B1--B9. 
\begin{proposition}[Convergence of backfitting]\label{thm:conv:of:backfitting:lc}   %\eqref{eq:general:case:1}--\eqref{eq:hatS}
Under Assumptions B1--B3, with probability tending to 1, there exists a unique solution $\{\bar \alpha_j : j=0,\dots,d\}$ to %\eqref{eq:backfit:LC:theory}. 
\eqref{eq:LC:backfiting:equation}. Moreover, there exist constants $0 < \gamma < 1$ and $c> 0$ such that, with probability tending to 1, it holds:
\[
\int \left[\bar \alpha_j^{[r]}(x_j) -\bar \alpha_j(x_j)\right]^2 E_j(x_j)\mathrm dx_j \leq c\gamma^{2r}  \left( 1 + \sum_{l=0}^d \int \left [ \bar \alpha_l^{[0]}(x_l) \right]^2 E_l(x_l) \mathrm dx_l \right),
\]
for $j=0,\dots,d$. The functions $\bar \alpha_{l}^{[0]}$ are the starting values of the backfitting algorithm. For $r>0$ the functions $\bar \alpha_{0}^{[r]},\dots,\bar \alpha_{d}^{[r]}$ are defined by equation \eqref{eq:backfit:LC}. 

Moreover, under the additional Assumption B5, with probability tending to 1, there exists a solution $\{ \bar \alpha_j^s :j=0,\dots,d \}$ of \eqref{eq:backfit:s:LC} that is unique for $s=A,B$, respectively . 
\end{proposition}

\begin{proposition}[Asymptotic behavior of stochastic part]\label{thm:stochastic:part:lc}
Suppose that Assumptions B1--B6 hold for a sequence $\Delta_n$ and intervals $S_j$, $j=0,\dots,d$. Then it holds that 
\[
\sup_{x_j\in S_j} \left\lvert \bar \alpha_j^A(x_j) - [\hat  \alpha_j^A(x_j) -\bar \alpha_{0,j}^A] \right\rvert = o_P(\Delta_n).
\]
Under the additional Assumption B7  it holds
\[
\sup_{x_j\in S_j} \left\lvert \bar \alpha_j^A(x_j) - [\hat  \alpha_j^A(x_j) -\bar \alpha_{0,j}^A + \mu_{n,j}(x_j)] \right\rvert = o_P(\Delta_n).
\]
\end{proposition}

For the convergence of the bias term, we need the following. 

\begin{description}
\item[B8] For all $j \neq k$, it holds 
\[
\sup_{x_j\in S_j} \int \left \lvert   \frac{\hat E_{j,k}(x_j,x_k)}{\hat E_j(x_j)\hat E_k(x_k)}  - \frac{ E_{j,k}(x_j,x_k)}{E_j(x_j)E_k(x_k)}    \right\rvert E_k(x_k)\mathrm dx_k = o_p(1). 
\]
\end{description}
At last, Assumption B9 is about the structure of the bias term of the estimators. 
\begin{description}
\item[B9] There exist deterministic functions $a_{n,0}(x_0), \dots , a_{n,d}(x_d)$ and constants $a_{n}^*$, $\gamma_{n,0},\dots,\allowbreak \gamma_{n,d}$ and a function $\beta:\mathbb R \to \mathbb R$ (not depending on $n$), such that 

\begin{align*}
\int a_{n,j}(x_j)^2 E_j(x_j)\mathrm dx_j &< \infty ,\\
\int \beta(x)^2 E(x) \mathrm d x  &< \infty ,\\
\sup_{x_1\in S_1, \dots, x_d\in S_d}  \lvert \beta(x) \rvert &< \infty ,\\
\gamma_{n,j} - \int a_{n,j}(x_j) \hat E_j(x_j) \mathrm dx_j  &=  o_P(\Delta_n) , \\
\sup_{x_j\in S_j} \left\lvert \hat \alpha_j^B(x_j) - \hat \mu_{n,0} -\hat \mu_{n,j}(x_j) \right\rvert &= o_P(\Delta_n), \\
\int \left\lvert \hat \alpha_j^B(x_j) - \hat \mu_{n,0} -\hat \mu_{n,j}(x_j) \right\rvert ^2 E_j(x_j) \mathrm dx_j &= o_P(\Delta_n^2), 
\end{align*}
for random variables $\hat \mu_{n,0}$ and where
\[
\hat \mu_{n,j}(x_j) = a_{n}^* + a_{n,j}(x_j) + \sum_{k\neq j} \int  a_{n,k}(x_k)  \frac{\hat E_{j,k}(x_j,x_k)}{\hat E_j(x_j)} \mathrm dx_k  + \Delta_n \int \beta(x)  \frac{E(x)}{E_j(x_j)}\mathrm d x_{-j}. 
\]
\end{description}

The following Proposition is taken from \cite{Mammen:etal:99} and we have adapted it to our notation. It implies in particular that the bias term of the smooth backfitting estimators equals the projections of the bias of the full-dimensional estimator of \cite{Linton:etal:03}. 

\begin{proposition}[Asymptotic behavior of bias part]\label{thm:bias:part:lc}
Under Assumptions B1--B6, B8, B9, for $j=0,\dots,d$,  it holds 
\begin{align*}
\sup_{x_j\in S_j} \left\lvert \bar \alpha_j^B(x_j) - \mu_{n,j}(X_j)  \right\rvert = o_P(\Delta_n),
\end{align*}
for $\mu_{n,j}(x_j) = a_{n,j}(x_j) - \gamma_{n,j} + \Delta_n \beta_j(x_j)$ with 
\[
(\beta_0, \beta_1, \dots,\beta_d) = \argmin_{\mathcal B} \int \left[ \beta(x) -\beta_0 - \beta_1(x_1) - \dots - \beta_d(x_d) \right]^2 E(x) \mathrm d x,
\]
and $\mathcal B = \{\tilde\beta=(\beta_0,\beta_1,\dots,\beta_d) : \int \beta_j(x_j)E_j(x_j)\mathrm dx_j = 0 ; j=0,\dots,d\}$. 
In particular, does Assumption B7 hold with this choice of $\mu_{n,j}(x_j)$. 
\end{proposition}

With the next lemma we ensure that the constant $\alpha^*$ is estimated at parametric rate in the local constant setting. This standard result will also be needed in the proof of Theorem \ref{thm:main:theorem:lc}. 

\begin{lemma}\label{lem:alpha:star:parametric:rate:LC} Let $\bar \alpha^* =  \left(\sum_{i=1}^n \int \mathrm dN_i(s) \right) / \left(\sum_{i=1}^n \int Y_i(s) \mathrm ds \right)$ as defined in equation \eqref{eq:alphastar:simple}. Under the condition $ \int \alpha_j(x_j) E_j(x_j) \mathrm dx_j = 0 $, for $j=0,\dots,d$ together with Assumption A2, it holds 
%\begin{center}  \textbf{Assumptions?}  \end{center}
\[
%\hat \alpha^* - \alpha^* = o_p(n^{-1/2}).
n^{1/2} \left (\bar \alpha^* - \alpha^*\right ) \to \mathcal N\left(0,\sigma_{\alpha^*}^2 \right),
\]
as $n\to \infty$ and for $\sigma_{\alpha^*}^2 =  \alpha^*(1- \alpha^* )$. This implies in particular $\bar \alpha^* - \alpha^* = O_p(n^{-1/2})$. 
\end{lemma}
\begin{proof}
We first note that it holds $E_0(t) = \int E(x) \mathrm dx_{-0} = \gamma(t)$ for $x=(t,z)$ and with $\gamma$ from Assumption A2. 
Using $\frac 1 n \sum_{i=1}^n Y_i(s) = \gamma(s) + o_P(1)$ in the denominator and the usual martingale decomposition for counting processes in the numerator, we get
\begin{align*}
\mathbb E\left[ n^{1/2}  \bar \alpha^* \right] &= n^{1/2} \alpha^*  + o(1), \\
\Var\left(n^{1/2} \bar \alpha^* \right) &= \alpha^*(1- \alpha^* ) + o(1),
\end{align*}
because of the identification $\int \alpha_0(s) \gamma(s) \mathrm ds =0$. The terms $ \mathbb E\left[ \int \alpha_{j}(Z_{i,j}(s)) \gamma(s) \mathrm ds\right]$ in the stable part of the martingale vanish because of $\gamma(t) = \int E(x) \mathrm dx_{-0}$ and the identification criterion. The Central Limit Theorem for \textit{i.i.d.}\ observations then yields the result. 
\end{proof}

%Moreover, we will make use of the following extension of the central limit theorem for martingales of \cite{RamlauHansen:83}.
%\begin{lemma}[Ramlau-Hansen]\label{lem:RamlauHansen} Let $\{M_i : i=1,\dots,n\}$ be a sequence of \textit{i.i.d.}\  martingales and let $g^{(n)}_i$  be predictable. Furthermore, suppose it holds 
%\begin{align}
%\sum_{i=1}^n \int  \left[g^{(n)}_i(s)\right]^2 \mathrm d\langle M_i \rangle (s) &\to \sigma^2 , \label{eq:ramlauhansen1}\\
%\sum_{i=1}^n \int  \left[g^{(n)}_i(s)\right]^2I_{\{\lvert g^{(n)}_i(s)\rvert > \varepsilon\}} \mathrm d\langle M_i \rangle (s) &\to 0, \label{eq:ramlauhansen2}
%\end{align}
%in probability for $n\to\infty$ with $\sigma^2>0$ and for every $\varepsilon > 0$. Then 
%\[
%\sum_{i=1}^n \int  g^{(n)}_i(s) \mathrm dM_i(s) \to \mathcal N(0,\sigma^2).
%\]
%\end{lemma}

%\begin{center} \textbf{ do we have to prove the following lemma?} \end{center}
Moreover, we will make use of the following counting process martingale central limit theorem, which is a direct application of Rebolledo's Theorem (Theorem II.5.1 in \cite{Andersen:etal:93}). % of the Lindeberg-Feller Theorem (see e.g.\ Proposition 2.27 in \cite{VanderVaart:00}) adapted to martingales. 
It is a multivariate extension of the central limit theorem for martingales in \cite{RamlauHansen:83}.

\begin{lemma}[Multivariate Ramlau-Hansen]\label{lem:multivar:RamlauHansen} Let $\{M_i : i=1,\dots,n\}$ be a sequence of \textit{i.i.d.}\ martingales and let $g^{(n)}_{i,j}$ be predictable functions for $j=1,\dots,d$. Furthermore, suppose it holds for $j,k=1,\dots,d$, 
\begin{align}
\sum_{i=1}^n  \int  g^{(n)}_{i,j}(s)g^{(n)}_{i,k}(s) \mathrm d\langle  M_i \rangle (s)  &\to \sigma_{j,k}^2 ,  \label{eq:multivar:ramlauhansen1}\\
%\sum_{i=1}^n \Cov \left( \int  \left[g^{(n)}_{i,j}(s)\right]^2 \mathrm d M_i  (s) , \int  \left[g^{(n)}_{i,k}(s)\right]^2 \mathrm d M_i  (s) \right) &\to \sigma_{j,k}^2 ,  \label{eq:multivar:ramlauhansen1}\\
%\sum_{i=1}^n \Cov \left( \int  \left[g^{(n)}_{i,j}(s)\right]^2 I_{\{\lvert g^{(n)}_{i,j}(s)\rvert > \varepsilon\}} \mathrm d\langle M_i \rangle (s) , \int  \left[g^{(n)}_{i,k}(s)\right]^2 I_{\{\lvert g^{(n)}_{i,k}(s)\rvert > \varepsilon\}} \mathrm d\langle M_i \rangle (s) \right) &\to 0 ,\label{eq:multivar:ramlauhansen2}
\sum_{i=1}^n  \int  \left[g^{(n)}_{i,j}(s)\right]^2 I_{\{\lvert g^{(n)}_{i,j}(s)\rvert > \varepsilon\}} \mathrm d\langle M_i \rangle (s) &\to 0 ,\label{eq:multivar:ramlauhansen2}
\end{align}
in probability for $n\to\infty$ with $\sigma_{j,k}^2>0$ and for every $\varepsilon > 0$. Then 
\[
\sum_{i=1}^n\left( \begin{matrix}  \int  g^{(n)}_{i,1} (s) \mathrm dM_i(s) \\ \vdots \\ \int  g^{(n)}_{i,d} (s) \mathrm dM_i(s) \end{matrix} \right) \to \mathcal N(0,\Sigma),
\]
in distribution for $n \to \infty$, where $\sigma_{j,k}^2$, $j,k=1,\dots,d$ are the entries of the covariance matrix $\Sigma$. 
\end{lemma}

To show Theorem \ref{thm:main:theorem:lc} we apply Propositions \ref{thm:conv:of:backfitting:lc}--\ref{thm:bias:part:lc} and Lemmas \ref{lem:alpha:star:parametric:rate:LC} and \ref{lem:multivar:RamlauHansen}.
According to the propositions it is sufficient to verify Assumptions  B1--B9. In the proof of Theorem \ref{thm:main:theorem:lc} we will show that our Assumptions A1--A5 imply Assumptions B1--B9 for the right choices of $\Delta_n, a_{n,j}, \beta, \gamma_{n,j}$.

\begin{proof}[Proof of Theorem \ref{thm:main:theorem:lc}]
 In the following we show how Assumptions A1--A5 imply B1--B6, B8--B9 with our choice of marginal pilot estimators. Assumption B7 is established through Proposition \ref{thm:bias:part:lc} once the other assumptions are verified. 

Without loss of generality, the proofs are done for $\mathcal{T}=R=1$, i.e.\ for survival time and covariates with support $[0,1]$ and we will show that Assumptions B1--B9 are satisfied on closed subsets $S_0 \subset (0,\mathcal{T})$ and $S_j \subset (0,R)$, $j=1,\dots,d$.

We first note that Assumption B1 follows directly from A1. 

For the remaining stochastic statements, we start with the derivation of convergence rates for the marginal exposure estimators. 
Moreover, we will show all statements for the rate $\Delta_n=h^2$. With $I_h= [2h,1-2h]$, %and $I_h^c=[0,1]\setminus I_h$, 
it holds for $j=0,\dots,d$, 
\begin{align}
\sup_{x_j \in I_h} &\lvert \hat E_{j}(x_j)  - E_{j}(x_j)  \rvert {=} O_P\left((\log n)^{1/2} n^{-2/5}\right)   \label{eq:sup:Ih:Ej},\\
\sup_{x_j, x_k \in I_h} &\lvert \hat E_{j,k}(x_j,x_k)  - E_{j,k}(x_j,x_k)  \rvert {=} O_P\left((\log n)^{1/2} n^{-3/10}\right), \label{eq:sup:Ih:Ejk}\\
\sup_{0\leq x_j \leq 1} &\lvert\hat E_{j}(x_j)  -  \int_0^1 k_h(x_j,u) \mathrm du  \  E_{j}(x_j)   \rvert {=} O_P\left(n^{-1/5}\right) , \label{eq:sup:01:Ej} \\
\sup_{0\leq x_j, x_k \leq1} &\lvert \hat E_{j,k}(x_j,x_k)  - \int_0^1 k_h(x_j,u) \mathrm du \int_0^1 k_h(x_k,v ) \mathrm dv \ E_{j,k}(x_j,x_k)  \rvert {=} O_P\left(n^{-1/5}\right) \label{eq:sup:01:Ejk}.
\end{align}
Before proving equations \eqref{eq:sup:Ih:Ej}--\eqref{eq:sup:01:Ejk} we emphasize that they imply in particular 
\begin{align}
\sup_{x_j \in [0,1]} \lvert \hat E_{j}(x_j)    \rvert &{=} O_P(1) \label{eq:Ej:Op:of:1} , \\
\sup_{x_j \in [0,1]} \lvert \hat E_{j}(x_j)  ^{-1}  \rvert &{=} O_P(1),     \label{eq:Ej:minus1:Op:of:1}  \\
\sup_{x_j, x_k \in [0,1]} \lvert \hat E_{j,k}(x_j,x_k)  \rvert &{=} O_P(1) .   \label{eq:Ejk:Op:of:1} 
\end{align}
Condition \eqref{eq:sup:Ih:Ej} follows with standard arguments (chaining, Bernstein inequality, c.f.\ \cite{Mammen:etal:99} for the regression case) from
\begin{align}
\mathbb E [ \hat E_j(x_j)] - E_j(x_j) &= O\left(n^{-2/5}\right), \label{eq:Ej:1} \\
\lvert \hat E_j(x_j) \rvert &\leq C_1 \ \ \ \ \ \ \text{a.s.}, \label{eq:Ej:2}\\
\lvert \hat E_j(u_1) - \hat E_j(u_2) \rvert &\leq  C_2 \lvert u_1 - u_2 \rvert n^m O_P(1), \label{eq:Ej:3}\\
\Var( \hat E_j(x_j)) &= O(n^{-4/5}),\label{eq:Ej:4}
\end{align}
for constants $0< C_1,C_2 <\infty$, $m >0$ and all $u_1\neq u_2,x_j\in[0,1]$. This can be seen with Taylor expansions and using the Lipschitz continuity of $K$. Condition \eqref{eq:sup:Ih:Ejk}--\eqref{eq:sup:01:Ejk}  can be shown in the same way. For  \eqref{eq:sup:01:Ej} and \eqref{eq:sup:01:Ejk} note that $\int_0^1 k_h(x_j,u) \mathrm du $ corrects the kernel at the boundaries where it does not integrate to unity.    %see notes. 
%Equations \eqref{eq:sup:01:Ej} and \eqref{eq:sup:01:Ejk} can be shown with a Taylor expansion of $\gamma$ from Assumption A2.    % $\gamma$????

We now show \eqref{eq:Ej:1}--\eqref{eq:Ej:4}. Condition \eqref{eq:Ej:2} follows directly from A3 with K being bounded and the covariates having compact support. 
With usual kernel estimator arguments and a Taylor expansion of $f_s$ around $x_j$ we get 
\begin{equation} \label{eq:Ej:op:of:hsquare}
\mathbb E [ \hat E_j(x_j)] - E_j(x_j)  = o(h^2), 
\end{equation}
which implies condition \eqref{eq:Ej:1} immediately. Condition \eqref{eq:Ej:4} can be derived analogously. Eventually, the Lipschitz continuity of $K$ in A3 yields \eqref{eq:Ej:3}.  

Since the kernel $k$ is cut off outside $[0,1]$, Assumption B2 follows directly from \eqref{eq:Ej:Op:of:1}--\eqref{eq:Ejk:Op:of:1}. \\

For the remaining assumptions we split the marginal estimator $\hat \alpha_j(x_j) $ as described for B5 into the variable part 
\[ 
 \hat \alpha_j^A(x_j) = \frac{  \frac 1 n \sum_{i=1}^n \int k_h(x_j,X_{ij}(s))  \mathrm d M_i(s)}{\hat E_j(x_j)},
\]
and the stable part 
\[ 
\hat \alpha_j^B(x_j) =     \frac{ \frac 1 n \sum_{i=1}^n \int k_h(x_j,X_{ij}(s)) \mathrm d \Lambda_{i}(s)}{\hat E_j(x_j)}, 
\]
via  $\hat \alpha_j(x_j) = \hat \alpha_j^A(x_j)  + \hat \alpha_j^B(x_j) $. % and we derive the asymptotic behavior of these terms separately. 
With the choice $\Lambda_i(t) = \int_0^t \lambda_i(s) \mathrm ds$ for the intensity $\lambda_i$ that was introduced in equation \eqref{eq:aalens:model}, we get that 
$M_i  = N_i -\Lambda_i$  defines a unique square integrable martingale arising from the counting process $N_i$. 

Next we derive the asymptotic behavior of $\hat \alpha_j^A(x_j)$ and $\hat \alpha_j^B(x_j)$ separately. With $M_i$ being a martingale and $k_h(x_j,X_{ij}(s))$ being predictable, the integral $ \int k_h(x_j,X_{ij}(s))  \mathrm d M_i(s)$ is a martingale as well. 
Using the multivariate  Ramlau-Hansen martingale central limit theorem in Lemma \ref{lem:multivar:RamlauHansen}, 
we will show that $\hat \alpha_j^A(x_j)$ is asymptotically normally distributed whereas the difference between the stable part $\hat \alpha_j^B(x_j)$ and $\alpha_j(x_j)$ asymptotically behaves like the bias term $b_j(x_j)$. 

%\begin{center} \textbf{change everything from here } \end{center}

For $x_j  \in I_h$, we now show conditions \eqref{eq:multivar:ramlauhansen1} and \eqref{eq:multivar:ramlauhansen2} of  Lemma \ref{lem:multivar:RamlauHansen} for $g_{ij}^{(n)}(s) = n^{-3/5} k_h(x_j - X_{ij}(s))$. Note that with $\Lambda_i$ being the compensator of $M_i$, we get in particular $d\langle M_i\rangle (s) = d\Lambda_i(s)= \left[ \alpha^* + \sum_{k=0}^d \alpha_k(X_{ik}(s))\right] Y_i(s) \mathrm ds $.  

For cross-terms with $j\neq l$ in \eqref{eq:multivar:ramlauhansen1}, it holds with this choice of $g_{ij}^{(n)}$ that
\begin{align}
\begin{split} \label{eq:martingale:alpha:tildeA1}
&\mathbb E \left[ \sum_{i=1}^n  \int  g^{(n)}_{i,j}(s)g^{(n)}_{i,k}(s) \mathrm d\langle  M_i \rangle (s) \right] \\
=& \mathbb E \left[ \left(\frac 1 n n^{2/5}\right)^2 \sum_{i=1}^n  \int k_h(x_j - X_{ij}(s)) k_h(x_l - X_{il}(s)) d\Lambda_i(s) \right] \\
=& n^{-1/5}\int \int   k_h(x_j - u_j) k_h(x_l - u_l)  \left[ \alpha^* + \alpha_0(s) +  \sum_{k=1}^d \alpha_k(u_k)\right] \\
&\hphantom{ n^{-1/5}\int \int }\times   \gamma(s) f_s(u_1,\dots,u_d) \mathrm d(u_1,\dots,u_d) \mathrm ds  \\
=& O(h),
\end{split}
\end{align}
because of the bounded support of the covariates and with the hazard rates being continuous. We write  $f_s(u_1,\dots,u_d)$ for the conditional density of $(X_{i1}(s), \dots, X_{id}(s))$ at $(u_1,\dots,u_d)$ given $Y_i(s)=1$. Moreover, it can be shown easily with similar arguments that the variance of these terms satisfies
\begin{equation}
\Var \left( \sum_{i=1}^n  \int  g^{(n)}_{i,j}(s)g^{(n)}_{i,k}(s) \mathrm d\langle  M_i \rangle (s) \right) = O(h^6),
\end{equation}
and hence $\sigma_{k,l}^2 = 0$ for $ j\neq l $ is assured for \eqref{eq:multivar:ramlauhansen1}. 
For the diagonal of the asymptotic covariance matrix $\tilde \Sigma$, we start with the following preliminary results. 
For $x_j  \in I_h$ it holds
\begin{align}
\begin{split} \label{eq:martingale:alpha:tildeA3}
&n^{4/5} \mathbb E \left[   n^{-2} \sum_{i=1}^n \int k_h(x_j - X_{ij}(s))^2 \alpha_j(X_{ij}(s)) Y_i(s) \mathrm ds  \right]\\
=& n^{4/5} n^{-1}\int \int k_h(x_j - u)^2 \alpha_j(u) f_s(u) \gamma(s) \mathrm du \, \mathrm ds \\
=& n^{-1/5} h^{-1} \int \int k(v)^2 \alpha_j(x_j+vh) f_s(x_j + vh) \gamma(s) \mathrm dv \, \mathrm ds  \\ 
=&(nh^5)^{-1/5}  \int k(v)^2\alpha_j(x_j) \mathrm dv E_j(x_j) +o(1) \\
= & c_h^{-1}\int k(v)^2   \mathrm dv \, \alpha_j(x_j) E_j(x_j) +o(1), 
\end{split}
\end{align}
with usual kernel estimator arguments. 
Analogously, we get for $l\neq j$, that 
\begin{align}
\begin{split} \label{eq:martingale:alpha:tildeA4}
&n^{4/5} \mathbb E \left[   n^{-2} \sum_{i=1}^n \int k_h(x_j - X_{ij}(s))^2 \alpha_l(X_{il}(s)) Y_i(s) \mathrm ds  \right] \\
%= & \int k(v)^2 \mathrm dv  \int \alpha_k(u_{k,s}) f_s(x_j,u_{k,s}) \gamma(s) \mathrm d u_{k,s} \mathrm d s  +o(1), 
= & c_h^{-1} \int k(v)^2 \mathrm dv\int  \int \alpha_k(u_l) f_s(x_j,u_{l}) \gamma(s) \mathrm d u_{l} \, \mathrm d s  +o(1).
\end{split}
\end{align}
For the variance of the diagonal terms, one can derive
\begin{equation}\label{eq:martingale:alpha:tildeA5}
\Var \left( \sum_{i=1}^n  \int  \left(g^{(n)}_{i,j}(s)\right)^2\mathrm d\langle  M_i \rangle (s) \right) = O(h^5),
\end{equation}
which yields the stochastic convergence of diagonal variance terms together with \eqref{eq:martingale:alpha:tildeA3} and \eqref{eq:martingale:alpha:tildeA4}. 
%where $ f_s(x_j,u_{l})$ is the conditional density of the pair $(X_{ij}(s), X_{ik}(s))$ at $(x_j,u_{l})$ given $Y_i(s)=1$
%Moreover, we get 
%\begin{equation}\label{eq:martingale:alpha:tildeA3}
%n^{4/5} \mathbb E \left[   n^{-2} \sum_{i=1}^n \int k_h(x_j - X_{ij}(s))k_h(x_l - X_{il}(s)) \alpha_l(X_{il}(s)) Y_i(s) \mathrm ds  \right] =  O(h). 
%\end{equation}

%Because of the identity
%\begin{align*}
%&n^{-3/5} \sum_{i=1}^n \int k_h(x_j - X_{ij}(s))k_h(x_l - X_{il}(s)) \mathrm d\langle M_i\rangle (s)    \\
%=&  n^{-3/5} \sum_{i=1}^n \int k_h(x_j - X_{ij}(s))k_h(x_l - X_{il}(s))  \left[ \alpha^* + \sum_{k=1}^d \alpha_k(X_{ik}(s))\right] Y_i(s) \mathrm ds ,
%\end{align*}

Summarizing, equations \eqref{eq:martingale:alpha:tildeA1}--\eqref{eq:martingale:alpha:tildeA5} imply condition \eqref{eq:multivar:ramlauhansen1} of  Lemma \ref{lem:multivar:RamlauHansen} with $\sigma_{j,j}^2 =\tilde \sigma_j^2(x_j)$ for %  where we use the asymptotic variance 
\[
%\tilde \sigma_j(x_j) =  \int k^2(v) \mathrm dv  \left( \alpha^* + \sum_{l \neq j}  \int \alpha_k(u_{k,s}) f_s(x_j,u_{k,s}) \gamma(s) \mathrm d u_{k,s} \mathrm d s  + \alpha_j(x_j) E_j(x_j)  \right). 
\tilde \sigma_j^2(x_j) =   c_h^{-1} \int k^2(v) \mathrm dv  \left( \alpha^* + \sum_{l \neq j}  \int\int \alpha_k(u_l) f_s(x_j,u_l) \gamma(s) \mathrm d u_l \, \mathrm d s  + \alpha_j(x_j) E_j(x_j)  \right), 
\]
and $\sigma_{j,k}^2=0$, $j\neq k$. 

The Lindeberg condition \eqref{eq:multivar:ramlauhansen2} is satisfied under Assumption A3 since we assume bounded support for all covariates. 

Hence, Lemma \ref{lem:multivar:RamlauHansen} implies 
\begin{equation}\label{eq:alpha:A:clt:temp}
n^{2/5} \left( \begin{matrix} \hat \alpha_0^A(x_0) \hat E_0(x_0) \\ \vdots\\ \hat \alpha_d^A(x_d)  \hat E_d(x_d) \end{matrix}\right) \to \mathcal N (0, \tilde\Sigma),
\end{equation}
where $\Sigma$ is a diagonal matrix with the entries $\tilde \sigma_j^2(x_j)$, $j=0,\dots,d$. 

Equations \eqref{eq:Ej:4} and \eqref{eq:Ej:op:of:hsquare}  imply convergence in probability of $\hat E_j(x_j)$ to $E_j(x_j)$ at a fast enough rate and hence, we get 
\begin{equation}\label{eq:alpha:A:clt}
n^{2/5} \left( \begin{matrix} \hat \alpha_0^A(x_0) \\ \vdots\\ \hat \alpha_d^A(x_d)  \end{matrix}\right) \to \mathcal N (0, \Sigma),
\end{equation}
from \eqref{eq:alpha:A:clt:temp} with  $\Sigma$ being a diagonal matrix with the entries $\sigma_j^2(x_j) = \tilde \sigma_j^2(x_j) E_j(x_j)^{-2}$, $j=0,\dots,d$. 

Note that condition \eqref{eq:alpha:A:clt}, implies in particular $\Var  \left(\hat \alpha_j^A(x_j) \right)  = O(n^{-4/5})$. 
Following the line of argumentation we used to prove \eqref{eq:sup:Ih:Ej} for $\hat E_j(x_j)$, this leads to 
\begin{equation}
\sup_{x_j \in I_h} \lvert \hat \alpha_j^A(x_j)  \rvert {=} O_P\left((\log n)^{1/2} n^{-2/5}\right)   \label{eq:sup:Ih:alphajA}.
\end{equation}
Analogously, one can get a similar result at the boundary and thus 
\begin{equation}
\sup_{x_j \in [0,1]} \lvert \hat \alpha_j^A(x_j) \rvert =  O_P \left(1\right) \label{eq:sup:01:alphajA} 
\end{equation}
on the whole support. 

For the stable part, we refer to \cite{Nielsen:Linton:95} who have shown for \[B_j(x_j) = \frac 1 n \sum_{i=1}^n \int k_h(x_j,X_{ij}(s)) \mathrm d \Lambda_i(s)\] that 
\begin{align}
\sup_{x_j \in [0,1]} \lvert B_j(x_j) - \mathbb E[  B_j(x_j) ] \rvert &= o_P(1), \label{eq:stable:part1} \\
\sup_{x_j \in [0,1]} \lvert \mathbb E[  B_j(x_j) ] \rvert &= o(1), \label{eq:stable:part2}
\end{align}
making use of the Lipschitz continuity of $K$ from Assumption A3 and of Assumption A1. 
Together with \eqref{eq:Ej:minus1:Op:of:1}, equations \eqref{eq:stable:part1}  and \eqref{eq:stable:part2}  imply
\begin{equation}
\sup_{x_j \in [0,1]} \lvert \hat \alpha_j^B(x_j) \rvert =  O_P \left(1\right). \label{eq:sup:01:alphajB} 
\end{equation}
One can get Assumptions B3 and B5 immediately from \eqref{eq:sup:01:alphajA}  and \eqref{eq:sup:01:alphajB}. 
Assumptions B2, B4 and B8 follow from equations \eqref{eq:sup:Ih:Ej}--\eqref{eq:sup:01:Ejk}. 

We illustrate the derivation of Assumption B6 for $x_j\in I_h$. First note that 
$ %\begin{equation*}
 \int  { E_{j,k}(x_j,x_k)}\allowbreak{ (E_j(x_j))^{-1}}\allowbreak  k_h(x_j-X_{i,j}(s)) \mathrm dx_j 
$ %\end{equation*}
is a bounded function $g(h,x_k,X_{i,j}(s))$ of arguments $h$, $x_k$, and $X_{i,j}(s)$ and hence predictable. This leads to 
\begin{equation*}
%\Var\left(  \int k_h(x_j-X_{ij}(s)) \hat E_j(x_j)^{-1} \mathrm dM_i(s)\right) {=} O(h^{-1}), 
\Var  \left( \int g(h,x_k,X_{i,j}(s))\mathrm d M_i(s) \right)  = O(1),
\end{equation*}
due to $M_i$ being a square integral martingale and a similar derivation to \eqref{eq:martingale:alpha:tildeA1}--\eqref{eq:martingale:alpha:tildeA5}. Thus, it holds that
\begin{equation*}
n^{1/2} \left(\frac 1 n \sum_{i = 1}^n  \int \int  \frac{ E_{j,k}(x_j,x_k)}{ E_j(x_j)} k_h(x_j,X_{i,j}(s)) \mathrm dx_j \, \mathrm d M_i(s)  \right) 
\end{equation*}
is asymptotically normally distributed and in particular
\begin{equation*}
\frac 1 n \sum_{i = 1}^n  \int \int  \frac{ E_{j,k}(x_j,x_k)}{ E_j(x_j)} k_h(x_j,X_{i,j}(s)) \mathrm dx_j \, \mathrm d M_i(s)  = O_P\left(n^{-1/2}\right).
\end{equation*}
Note that by integrating over $x_k$, we achieve the parametric rate $n^{1/2}$ making the usual rate $h^{-1}$ vanish. 
Together with \eqref{eq:sup:Ih:Ej} and \eqref{eq:sup:Ih:Ejk}, the last equation yields 
\begin{align*}
& \int  \frac{\hat E_{j,k}(x_j,x_k)}{\hat E_k(x_k)} \hat \alpha_j^A(x_j) \mathrm dx_j   \\
% = & \int  \left ( \frac{ E_{j,k}(x_j,x_k)}{ E_k(x_k)} + O_P(n^{-3/10}n^{-2/5}\log{n}) \right) \tilde \alpha_j^A(x_j) \mathrm dx_j        \\
 = & \int  \frac{ E_{j,k}(x_j,x_k)}{ E_k(x_k)}  \hat \alpha_j^A(x_j) \mathrm dx_j   +   O_P(n^{-3/10}n^{-2/5}\log{n})   \\
= & E_k(x_k)^{-1} \frac 1 n \sum_{i = 1}^n   \int \int  \frac{ E_{j,k}(x_j,x_k)}{ E_j(x_j)} k_h(x_j,X_{i,j}(s)) \mathrm dx_j \, \mathrm d M_i(s)  + O_P(n^{-3/10}n^{-2/5}\log{n})\\
 = &   O_P\left(n^{-1/2}\right),
\end{align*}
since  \eqref{eq:sup:Ih:Ej} further implies $ \hat \alpha_j^A(x_j)=  E_j(x_j)^{-1}_h(x_j-X_{i,j}(s)) \mathrm d M_i(s) + O_P\left(n^{-2/5}(\log{n})^{1/2} \right)$. 

The last equation proves Assumption B6. 

We prove Assumption B9 for the following choices for $j=0,\dots,d$. 
\begin{align*}
a_n^* &= \alpha^*, \\
a_{n,j}(x_j) &= \alpha_j(x_j) + \alpha_j'(x_j) \int k_h(x_j,u) (u-x_j)  \left[ \int k_h(x_j,v)\mathrm d v\right] ^{-1}   \mathrm du, \\
\beta(x) &= \sum_{j=0}^d \left[ \alpha_j'(x_j) \frac{\partial \log E(x)}{\partial x_j }  + \frac 1 2 \alpha_j''(x_j) \right] \int u^2 k(u) \mathrm d u, \\
\gamma_{n,j} &= 0.
\end{align*}
The first three statement of B9 hold immediately with this choice of $a_{n,j}$ and Assumptions A1 and A3. 

For the fourth statement it holds
\begin{equation}
\int a_{n,j}(x_j) \hat E_j(x_j) \mathrm dx_j = \int \alpha_j(x_j) \hat E_j(x_j) \mathrm dx_j + \int \alpha_j'(x_j) \hat E_j(x_j) \frac{\int k_h(x_j,u) (u-x_j) }{\int k_h(x_j,v)\mathrm d v} \mathrm dx_j ,
\end{equation}
and we investigate the two summands separately. 
For the first one it holds 
\begin{align*}
\int \alpha_j(x_j) \hat E_j(x_j) \mathrm dx_j  =& \frac 1 n \sum_{i=1}^n \int \int \alpha_j(x_j) k_h(x_j,X_{ij}(s)) \mathrm dx_j Y_i(s) \mathrm ds \\
=& \frac 1 n \sum_{i=1}^n \int g_h(X_{i,j}(s)) Y_i(s) \mathrm ds  \\
%=& O_P\left(n^{-1/2}\right) \\
=& \mathbb E \left[ \int \alpha_j(x_j) \hat E_j(x_j) \mathrm dx_j   \right] + o_P\left(n^{-1/2}\right) \\
=& \int \int \int \alpha_j(x_j) k_h(x_j - u) \gamma(s)  f_s(u) \mathrm du \, \mathrm ds \, \mathrm dx_j  + o_P\left(n^{-1/2}\right)  \\
=&  \int \int \alpha_j(x_j) k_h(x_j - u)  E_j(u) \mathrm du \, \mathrm dx_j  + o_P\left(n^{-1/2}\right) \\
=&  \int \alpha_j(x_j)  E_j(x_u) \mathrm dx_j  + o_P\left(n^{-1/2}\right), 
\end{align*}
since $\int g_h(X_{i,j}(s)) Y_i(s) \mathrm ds$ are \textit{i.i.d.}\ random variables with the definition $g_h(X_{i,j}(s)) =\int \alpha_j(x_j) k_h(x_j-X_{ij}(s)) \mathrm dx_j$ and the Central Limit Theorem applies as for B6. The last equation follows from a substitution, a Taylor expansion of $E_j$ and the fact that $k$ is a kernel of order one. 

The second summand can be treated analogously yielding
\begin{align*}
& \int \alpha_j'(x_j) \hat E_j(x_j) \frac{\int k_h(x_j,u) (u-x_j) }{\int k_h(x_j,v)\mathrm d v} \mathrm dx_j \\
= & \int \int \alpha'_j(x_j) k_h(x_j - u) (u-x_j)E_j(u)  \mathrm du \, \mathrm dx_j   + o_P\left(n^{-1/2}\right),\\
=&   o_P\left(n^{-1/2}\right),
\end{align*}
and hence in total
\begin{equation}
\int a_j(x_j) \hat E_j(x_j) \mathrm dx_j = o_P\left(n^{-1/2}\right).  \label{eq:part:of:B9}
\end{equation}
because of the identification $  \int \alpha_j(x_j)  E_j(x_u) \mathrm dx_j = 0$. This verifies the fourth statement of B9 with $\gamma_{n,j} =0$. 

To prove B9, we start with two preliminary results: 
\begin{align}
\sup_{x_j \in I_h} \lvert \hat \alpha_j^B(x_j) - \hat \mu_{n,j}(x_j)  \rvert &{=} o_P\left(h^2 \right) \label{eq:bias:interior}, \\
\sup_{x_j \in I_h^c} \lvert \hat \alpha_j^B(x_j) - \hat \mu_{n,j}(x_j) \rvert &{=} o_P\left(h \right) \label{eq:bias:boundary}. 
\end{align}
Recall that by definition it holds 
\begin{align*}
\hat \alpha_j^B(x_j)  =&  \frac 1 n \sum_{i=1}^n \int k_h(x_j-X_{ij}(s)) \mathrm d \Lambda_i(s) \left(\hat E_j(x_j)\right)^{-1}  \\ 
 =&  \frac 1 n \sum_{i=1}^n \int k_h(x_j-X_{ij}(s))\left[\alpha^* + \sum_{l=0}^d  \alpha_l(X_{il}(s)) \right] Y_i(s) \mathrm d s \left(\hat E_j(x_j)\right)^{-1},
\end{align*}
and
\begin{align*}
 \hat \mu_{n,j}(x_j) =&  a_{n,0} + a_{n,j}(x_j) + \sum_{k\neq j} \int a_{n,k}(x_k) \frac{\hat E_{j,k}(x_j,x_k)}{\hat E_j(x_j)} \mathrm dx_k - \Delta_n \int \beta(x) \frac{E(x)}{E_j(x_j)} \mathrm dx_{-j} \\
=&\alpha^* + \alpha_j(x_j) + \alpha_j'(x_j) \int k_h(x_j,u) (u-x_j)  \left[ \int k_h(x_j,v)\mathrm d v\right] ^{-1}   \mathrm du \\
& + \sum_{k\neq j} \int \left( \alpha_k(x_k) + \alpha_k'(x_k) \int k_h(x_k,u) (u-x_k)  \left[ \int k_h(x_k,v)\mathrm d v\right] ^{-1}   \mathrm du \right) \\
 &\hphantom{+ \sum_{k\neq j} \int }\times \frac{\hat E_{j,k}(x_j,x_k)}{\hat E_j(x_j)} \mathrm dx_k \\
&+ \Delta_n \int u^2 k(u) \mathrm d u \int  \sum_{j=0}^d \left[ \alpha_j'(x_j) \frac{\partial \log E(x)}{\partial x_j }  + \frac 1 2 \alpha_j''(x_j) \right]  \frac{E(x)}{E_j(x_j)} \mathrm dx_{-j}.
\end{align*}
Next, it holds for $j=0,\dots,d$, 
\begin{align}
\begin{split}
&\frac 1 n \sum_{i=1}^n \int k_h(x_j,X_{ij}(s)) \alpha_j(X_{ij}(s))  Y_i(s) \mathrm d s \left(\hat E_j(x_j)\right)^{-1}\\
=& \alpha_j(x_j) + \alpha_j'(x_j) \int k_h(x_j,u) (u-x_j) \mathrm du \left( \int k_h(x_j,u)\mathrm du\right)^{-1} \\
& + h^2 \int u^2 k(u)\mathrm d u  \left[ E_j'(x_j) \alpha_j'(x_j) + \frac 1 2 E_j(x_j) \alpha''_j(x_j) \right] E_j(x_j)^{-1} + R_{n,j}(x_j)  \label{eq:bias:j:for:B9}, 
\end{split}
\end{align}
with $\sup_{x_j\in I_h}  \lvert R_{n,j}(x_j) \rvert = o_p(h^2)$ and $\sup_{x_j\in [0,1]\setminus I_h}  \lvert R_{n,j}(x_j) \rvert = O_p(h^2)$. Similarly, for $k  \neq j$, we get
\begin{align}
\begin{split}
&\frac 1 n \sum_{i=1}^n \int k_h(x_j,X_{ij}(s)) \alpha_k(X_{ik}(s))  Y_i(s) \mathrm d s \left(\hat E_j(x_j)\right)^{-1}\\
=& \int \alpha_k(x_k) \frac{\hat E_{j,k}(x_j,x_k)}{\hat E_j(x_j)} \mathrm dx_k \\
&+ \int  \alpha_k'(x_k)  \frac{\hat E_{j,k}(x_j,x_k)}{\hat E_j(x_j)}     k_h(x_k,u) (u-x_k)  \mathrm du \left( \int k_h(x_j,u)\mathrm du\right)^{-1} \\
& + h^2 \int u^2 k(u)\mathrm d u \int \left[ \frac{ \partial E_{j,k}(x_j,x_k)}{\partial x_k} \alpha_k'(x_k) + \frac 1 2 E_{j,k}(x_j,x_k) \alpha''_j(x_j) \right] E_j(x_j)^{-1} \\
&+ R_{n,j,j}(x_j), \label{eq:bias:jk:for:B9}
\end{split}
\end{align}
with $\sup_{x_j\in I_h}  \lvert R_{n,j,k}(x_j) \rvert = o_p(h^2)$ and $\sup_{x_j\in [0,1]\setminus I_h}  \lvert R_{n,j,k}(x_j) \rvert = O_p(h^2)$. 
Equation \eqref{eq:bias:j:for:B9} follows straightforward with a Taylor expansion of each $\alpha_j$ and $E_j$ and for the derivation of \eqref{eq:bias:jk:for:B9} we refer to the proof of Theorem 4 in \cite{Mammen:etal:99}, where the analogue is shown for the  nonparametric regression case. Equations \eqref{eq:bias:j:for:B9} and \eqref{eq:bias:jk:for:B9} imply \eqref{eq:bias:interior} and \eqref{eq:bias:boundary} with above choices of $a_{n,j},\beta$ and $\gamma_{n,j}$. Eventually, together with  \eqref{eq:part:of:B9}, conditions \eqref{eq:bias:interior} and \eqref{eq:bias:boundary} imply A9. 

%\begin{center} \textbf{show detailed derivation of \eqref{eq:bias:jk:for:B9}?} \end{center}
%Note that Proposition \ref{thm:bias:part:lc} implies that the bias is the sum of the marginal projections of the bias of the multivariate pilot estimator. 
For the last statement of the theorem, we note that the constant component $\alpha^*$ in the conditional hazard can be estimated at a parametric rate $n^{-1/2}$ by $\bar \alpha^*$ due to Lemma \ref{lem:alpha:star:parametric:rate:LC}.
\end{proof}

\subsection{Asymptotic theory for the local linear estimator}  \label{sec:theory:ll}

For the local linear estimator, we follow the same procedure as in Section \ref{sec:theory:lc}. 
We first introduce general assumptions as well as a set of results from \cite{Mammen:etal:99} which we will apply to prove Theorem \ref{thm:main:theorem:ll}. 
Then we verify the new assumptions under Assumptions A1--A5. 

Let $E:\mathcal X \to [0,1]$ be the exposure as defined earlier and let $W$ be a (deterministic) positive definite $(d+1)\times(d+1)$-matrix with elements $W_{r,s}$ such that $W_{0,0}=1$. We set 
\begin{align}
M_j(x_j) &= \left(\begin{matrix}W_{0,0} & W_{j,0} \\ W_{j,0} & W_{j,j}  \end{matrix}\right) E_j(x_j), \\
S_{l,j}(x_l,x_j) &= \left(\begin{matrix}W_{0,0} & W_{l,0} \\ W_{j,0} & W_{l,j}  \end{matrix}\right) E_{l,j}(x_l,x_j).
\end{align}
These will later be the fixed but unknown matrices to which $\hat M_j$ and $\hat S_j$, respectively, converge. 

Now we make the following assumptions which are all of similar nature to B1--B9. 
Note that these are assumptions on $\hat V_{}^j (x_j)$, $\hat V_{j}^j (x_j)$, $\hat V_{j}^j (x_j)$, $\hat V_{j,j}^j (x_j)$, $\hat V_{}^{l,j} (x_l,x_j)$, $\hat V_{l}^{l,j}(x_l,x_j)$, $\hat V_{j}^{l,j} (x_l,x_j)$, $\hat V_{l,j}^{l,j}(x_l,x_j)$ and $\hat \alpha_j(x_j), \hat \alpha^j (x_j)$, and all $x_j,x_l$,  $j,l=0,\dots,d$
and we don't assume any particular definition of these terms for the following propositions. 
\begin{description}
\item[B1'] %The one-dimensional marginals $E_{j}(x_j) = \int E(x) \mathrm dx_{-j}$ and the two-dimensional marginals $E_{j,k}(x_j,x_k) = \int E(x) \mathrm dx_{-(j,k)}$ of the exposure are strictly positive everywhere on the support and it holds
For all $ j\neq k$ it holds 
\[
\int \frac{E_{j,k}(x_j,x_k)^2}{E_j(x_j)E_k(x_k)} \mathrm d x_j \mathrm d x_k < \infty.
\]
\item[B2'] For $\hat M_j$ and $\hat S_{l,j}$ as in \eqref{eq:hatM} and \eqref{eq:hatS} it holds 
\begin{align*}
\int \left[\frac{\hat V_{}^j(x_j) - E_j(x_j)}{E_j(x_j)}\right]^2 E_j(x_j)\mathrm dx_j &= o_P(1), \\
\int \left[\frac{\hat V_{}^{j,k}(x_j,x_k)}{E_j(x_j)E_k(x_k)}  - \frac{E_{j,k}(x_j,x_k)}{E_j(x_j)E_k(x_k)}\right]^2 E_j(x_j)E_k(x_k)\mathrm dx_j  \, \mathrm dx_k & = o_P(1), \\
\int \left[\hat M_j(x_j)^{-1} \hat S_{k,j}(x_k,x_j)- M_j(x_j)^{-1} S_{k,j}(x_k,x_j) \right]_{r,s}^2 E_j(x_j)E_k^{-1}(x_k)\mathrm dx_j \,  \mathrm dx_k & = o_P(1),
\end{align*}
for $r,s=1,2$. Here $[A]_{r,s}$ denotes the element $(r,s)$ of a matrix $A$. Moreover, $\hat M_j$ vanishes outside the support of $E_j$, $\hat S_{j,k}$ vanishes outside the support of $E_{j,k}$ and $\hat S$ is symmetric, i.e. $\hat S_{j,k}(x_j,x_k)^T = \hat S_ {k,j}(x_k, x_j)$. 
\item[B3'] There exists a constant $C$ such that with probability tending to 1 for all $j$, 
\[
\int \hat \alpha_j(x_j)^2E_j(x_j) \mathrm dx_j \leq C, 
\]
and 
\[
\int \hat \alpha^j(x_j)^2E_j(x_j) \mathrm dx_j \leq C. 
\]
\item[B4']  For some finite intervals $S_j\subset \mathbb R$ that are contained in the support of $E_j$, $j=0,\dots,d$, we suppose that there exists a finite constant $C$ such that with probability tending to 1 for all $j\neq k$,
\[
\sup_{x_j\in S_j} \int \trace \left[ \hat S_{k,j}(x_k,x_j)\hat M_j(x_j)^{-2}\hat S_{k,j}(x_k,x_j)\right ] E_k(x_k)^{-1}\mathrm dx_k \leq C. 
\]
\end{description}

We now introduce the notation $\hat \alpha_j = \hat \alpha_j^A + \hat \alpha_j^B$ and $\hat \alpha^{j}=\hat \alpha^{j,A}+\hat \alpha^{j,B}$. Where $(\hat \alpha_j^A,\hat \alpha^{j,A})$ is the variable part and $(\hat \alpha_j^B,\hat \alpha^{j,B})$ is the stable part of the initialization $(\hat \alpha_j,\hat \alpha^{j})$. The terms are given by
\begin{align*} 
\hat \alpha_j^A(x_j) = &\left\{  (\hat V_{j}^j(x_j))^2 - \hat V_{j,j}^j(x_j)\hat V_{}^j (x_j)\right\}^{-1}   \frac 1 n \sum_{i=1}^n \int g_{i,j}(x_j)  \mathrm d M_i(s) , \\
\hat \alpha^{j,A} (x_j) = &\left\{  (\hat V_{j}^j(x_j))^2 - \hat V_{j,j}^j(x_j)\hat V_{}^j (x_j)\right\}^{-1}  \frac 1 n \sum_{i=1}^n \int g_{i}^j(x_j)  \mathrm d M_i(s), \\
\hat \alpha_j^B(x_j) = &\left\{  (\hat V_{j}^j(x_j))^2 - \hat V_{j,j}^j(x_j)\hat V_{0,0}^j (x_j)\right\}^{-1}   \frac 1 n \sum_{i=1}^n \int  g_{i,j}(x_j) \mathrm d \Lambda_i(s) , \\
\hat \alpha^{j,B} (x_j) = &\left\{  (\hat V_{j}^j(x_j))^2 - \hat V_{j,j}^j(x_j)\hat V_{0,0}^j (x_j)\right\}^{-1} \frac 1 n \sum_{i=1}^n \int g_{i}^j(x_j)  \mathrm d \Lambda_i(s) ,
\end{align*}
with 
\begin{align*} 
g_{i,j}(x_j) &=  \left[ \hat V_{j}^j(x_j) \left(\frac{x_j- X_{ij}(s)}{h}\right)  - \hat V_{j,j}^j(x_j) \right]k_h(x_j-X_{ij}(s)), \\
 g_{i}^j(x_j) &=   \left[ \hat V_{j}^j(x_j) - \hat V_{}^j(x_j) \left(\frac{x_j- X_{ij}(s)}{h}\right) \right]k_h(x_j-X_{ij}(s)). 
\end{align*}
Equivalently, we can write 
\begin{align*} 
\left( \begin{matrix} \hat \alpha_j^A(x_j)  \\ \hat \alpha^{j,A} (x_j)   \end{matrix} \right) &=   \frac 1 n \sum_{i=1}^n \int \left( \begin{matrix} 1 \\ h^{-1} (x_j - X_{ij}(s))\end{matrix}\right)  k_h(x_j,X_{ij}(s))  \mathrm d M_i(s), \\
\left( \begin{matrix} \hat \alpha_j^B(x_j)  \\ \hat \alpha^{j,B} (x_j) \end{matrix} \right) &=   \frac 1 n \sum_{i=1}^n \int \hat M_j(x_j)^{-1} \left( \begin{matrix} 1 \\ h^{-1} (x_j - X_{ij}(s))\end{matrix}\right)  k_h(x_j,X_{ij}(s))  \mathrm d \Lambda_i(s),
\end{align*}
As in Assumption B4, $M_i$ is the martingale arising from $N_i$ and $\Lambda_i$ is its compensator. 
Later on, we will verify the following assumptions on $(\hat \alpha_j^A,\hat \alpha^{j,A})$ and $(\hat \alpha_j^B,\hat \alpha^{j,B})$. Moreover, for the whole estimator we define, for $s\in \{A,B\}$, $\tilde \alpha_{0,j}^s$, $\tilde \alpha_j^s$ and $\tilde \alpha^{j,s}$ as the solution of the equations 
\begin{align} \label{eq:backfitalgoAA1}
\hat M_j(x_j) 
 \left( \begin{matrix}  \tilde \alpha^s_j(x_j) -\hat \alpha^s_j(x_j)  \\ \tilde \alpha^{j,s}(x_j) - \hat \alpha^{j,s}(x_j)  \end{matrix} \right)
&= \-\tilde \alpha^s_{0,j}   \left(\begin{matrix} \hat V_{}^j (x_j) \\  \hat V_{j}^j (x_j) \end{matrix}\right)
 - \sum_{l\neq j} \int \hat S_{l,j}(x_l, x_j)  \left(\begin{matrix} \tilde \alpha_l^s(x_l) \\ \tilde \alpha^{l,s}(x_l) \end{matrix}\right) \mathrm dx_l, \\
\int \tilde \alpha_j^s(x_j) \hat V_{}^j (x_j) \mathrm dx_j &= 0. \label{eq:backfitalgoAA2}
\end{align}
Existence and uniqueness of $\tilde \alpha_j^A, \tilde \alpha_j^B, \tilde \alpha^{j,A},\tilde \alpha^{j,B}$ is stated in Proposition \ref{thm:conv:of:backfitting:ll}.  We make the further assumptions
\begin{description}
\item[B5'] There exists a constant $C$ such that with probability tending to 1 for all $j$, it holds
\[
\int \hat \alpha_j^s(x_j)^2 E_j(x_j) \mathrm dx_j \leq C, 
\]
and
\[
\int \hat \alpha^{j,s}(x_j)^2 E_j(x_j) \mathrm dx_j \leq C,
\]
for $s=A,B$. 
\item[B6'] We assume that there is a sequence $\Delta_n$ such that 
\begin{align*}
\sup_{x_k\in S_k} \left\lVert  \int \hat M_k(x_k)^{-1}\hat S_{k,j}(x_k,x_j) \left( \begin{matrix} \hat \alpha_j^A(x_j) \\ \hat \alpha^{j,A}(x_j) \end{matrix} \right)\mathrm dx_j   \right\rVert_2 = o_P(\Delta_n), \\
\left\lVert  \int \hat M_k(x_k)^{-1}\hat S_{k,j}(x_k,x_j) \left( \begin{matrix} \hat \alpha_j^A(x_j) \\ \hat \alpha^{j,A}(x_j) \end{matrix} \right)\mathrm dx_j   \right\rVert_{M_k,2} = o_P(\Delta_n),
\end{align*}
where $\lVert \cdot \rVert_2$ denotes the $L_2$ norm in $\mathbb R^2$ and where for functions $g:\mathbb R\to \mathbb R^2$ we define $\lVert g\rVert^2_{M_k,2}= \int g(u)M_k(u)g(u) \mathrm du$. The sets $S_k$ have been introduced in Assumption B4'. % B4' 
\item[B7'] There exist deterministic functions $\mu_{n,j}$ such that 
\[
\sup_{x_j \in S_j} \left\lvert \tilde \alpha_j^B(x_j) - \mu_{n,j}(x_j) \right \rvert = o_p (\Delta_n),
\]
where  $S_k$ has been introduced in Assumption B4'. % B4' 
\end{description}

The local linear equivalents to Propositions \ref{thm:conv:of:backfitting:lc} and \ref{thm:stochastic:part:lc} are the following results from \cite{Mammen:etal:99}, adapted to our setting. The following two propositions assure convergence of the backfitting algorithm and asymptotic normality of the stochastic part of the estimator under Assumptions B1'--B7'. 
\begin{proposition}[Convergence of backfitting]\label{thm:conv:of:backfitting:ll}   %\eqref{eq:general:case:1}--\eqref{eq:hatS}
Under Assumptions B1'--B3'%B4'
, with probability tending to 1, there exists a unique solution $\{\tilde m_{0,l}, \tilde m_l, \tilde m^l : l=0,\dots,d\}$ to \eqref{eq:general:case:1}--\eqref{eq:hatS}. Moreover, there exist constants $0 < \gamma < 1$ and $c> 0$ such that, with probability tending to 1, it holds:
\begin{align*}
\int \left[\tilde \alpha_j^{[r]}(x_j) -\tilde \alpha_j(x_j)\right]^2 E_j(x_j)\mathrm dx_j &\leq c\gamma^{2r} \Gamma, \\
\int \left[\tilde \alpha^{j,[r]}(x_j) -\tilde \alpha^j(x_j)\right]^2 E_j(x_j)\mathrm dx_j &\leq c\gamma^{2r} \Gamma,
\end{align*}
where 
\[
\Gamma = 1 + \sum_{l=0}^d \int \left [ \tilde \alpha_l^{[0]}(x_l) \right]^2 E_l(x_l) \mathrm dx_l + \int \left [ \tilde \alpha^{l,[0]}(x_l) \right]^2 E_l(x_l) \mathrm dx_l.
\]
The functions $\tilde \alpha_{0,l}^{[0]}$, $\tilde \alpha_{l}^{[0]}$ and $\tilde \alpha^{l,[0]}$ are the starting values of the backfitting algorithm. For $r>0$ the functions $\tilde \alpha_{l}^{[r]}$ and $\tilde \alpha^{l,[r]}$ are defined by equations \eqref{eq:backfitalgo1} and \eqref{eq:backfitalgo2}. 

Moreover, under the additional Assumption B5', with probability tending to 1, there exists a solution $\{ \tilde \alpha _0^s, \tilde \alpha_j^s, \tilde \alpha^{j,s}:j=0,\dots,d \}$ of \eqref{eq:backfitalgoAA1}, \eqref{eq:backfitalgoAA2} that is unique for $s=A,B$, respectively. 
\end{proposition}

\begin{proposition}[Asymptotic behavior of stochastic part]\label{thm:stochastic:part:ll}
Suppose that Assumptions B1'--B6' hold for a sequence $\Delta_n$ and intervals $S_j$, $j=0,\dots,n$. Then it holds that 
\[
\sup_{x_j\in S_j} \left\lvert \tilde \alpha_j^A(x_j) - [\hat \alpha_j^A(x_j) -\tilde \alpha_{0,j}^A] \right\rvert = o_P(\Delta_n).
\]
Under the additional Assumption B7' it holds
\[
\sup_{x_j\in S_j} \left\lvert \tilde \alpha_j(x_j) - [\hat \alpha_j^A(x_j) -\tilde \alpha_{0,j}^A + \mu_{n,j}(x_j)] \right\rvert = o_P(\Delta_n).
\]
\end{proposition}

Before stating a result for the bias part, we assume the following. 

\begin{description}
\item[B8'] For all $j \neq k$, it holds 
\[
\sup_{x_j\in S_j} \int \left \lvert\left[ \hat M_j(x_j)^{-1} \hat S_{k,j}(s_k,x_j)- M_j^{-1}(x_j) S_{k,j}(x_k,x_j)\right]_{r,s}\right\rvert E_k(x_k)\mathrm dx_k = o_p(1),
\]
for $r,s=1,2$. 
\item[B9'] There exist deterministic functions $a_{n,0}(x_0), \dots , a_{n,d}(x_d),a_n^0(x_0),\dots,a_n^d(x_d)$ and constants $a_{n}^*$, $\gamma_{n,0},\dots,\gamma_{n,d}$ such that 
\begin{align*}
\int a_{n,j}(x_j)^2 E_j(x_j)\mathrm dx_j &< \infty ,\\
\int a_{n}^j(x_j)^2 E_j(x_j)\mathrm dx_j &< \infty , \\
\gamma_{n,j}  - \int a_{n,j}(x_j) \hat V^j_{}(x_j) \mathrm dx_j &=  o_P(\Delta_n) , \\
\sup_{x_j\in S_j} \left\lvert \tilde \alpha_j^B(x_j) - \hat \mu_{n,0} -\hat \mu_{n,j}(x_j) \right\rvert &= o_P(\Delta_n), \\
\int \left\lvert \hat \alpha_j^B(x_j) - \hat \mu_{n,0} -\hat \mu_{n,j}(x_j) \right\rvert ^2 E_j(x_j) \mathrm dx_j &= o_P(\Delta_n^2)\\
\sup_{x_j\in S_j} \left\lvert \hat \alpha^{j,B}(x_j) - \hat \mu_{n,0} -\hat \mu_{n}^j(x_j) \right\rvert &= o_P(\Delta_n), \\
\int \left\lvert \hat \alpha^{j,B}(x_j)  -\hat \mu_{n}^j(x_j) \right\rvert ^2 E_j(x_j) \mathrm dx_j &= o_P(\Delta_n^2),  
\end{align*}
for random variables $\hat \mu_{n,0}$ and where
\[
\left( \begin{matrix} \hat \mu_{n,j}(x_j) \\ \hat \mu_n^j(x_j) \end{matrix}\right) = \left(\begin{matrix} a_{n,0} + a_{n,j}(x_j) \\ a_n^j(x_j) \end{matrix}\right) + \sum_{k\neq j} \int \hat M_j(x_j) ^{-1} \hat S_{k,j}(x_k,x_j) \left(\begin{matrix} a_{n,k}(x_k) \\ a_n^k(x_k) \end{matrix}\right) \mathrm dx_k. 
\]
\end{description}

The next proposition appears in \cite{Mammen:etal:99} with different notation for the nonparametric regression case. It assures convergence of the deterministic part of the estimator. 

\begin{proposition}[Asymptotic behavior of bias part]\label{thm:bias:part:ll}
Under Assumptions B1'--B6', B8', B9', it holds 
\begin{align*}
\sup_{x_j\in S_j} \left\lvert \tilde \alpha_j^B(x_j) - \mu_{n,j}(X_j)  \right\rvert = o_P(\Delta_n), \\
\sup_{x_j\in S_j} \left\lvert \tilde \alpha^{j,B}(x_j) - \mu_{n}^j(X_j)  \right\rvert = o_P(\Delta_n),
\end{align*}
for $\mu_{n,j}(x_j) = a_{n,j}(x_j) - \gamma_{n,j}$ and $\mu_n^j(x_j) = a_n^j(x_j)$. Assumption B7' holds with this choice of $\mu_{n,j}(x_j)$. 
\end{proposition}

%Analogously to Lemma \ref{lem:alpha:star:parametric:rate:LC}, we now establish parametric convergence of the estimator of the constant component in the local linear case in the following lemma and we will use it in the proof of Theorem \ref{thm:main:theorem:ll}. 
%\begin{lemma}\label{lem:alpha:star:parametric:rate:LL} Let $\tilde \alpha^{*} = \frac { \sum_{i=1}^n  \int \mathrm dN_i(s) } {\sum_{i=1}^n \int Y_i(s) \mathrm  ds}  - \sum_{j=0}^d \left(\int  \hat V^j_{j,0}(x_j) \mathrm dx_j \right)^{-1} \int \tilde \alpha^j(x_j) \hat V^j_{j,0}(x_j) \mathrm dx_j $  as defined in equation \eqref{eq:alphastar:LL}. Under the condition $ \int \alpha_j(x_j) E_j(x_j) \mathrm dx_j = 0 $, for $j=0,\dots,d$ together with Assumption A2, it holds 
%%\begin{center}  \textbf{Assumptions?}  \end{center}
%\[
%%\hat \alpha^* - \alpha^* = o_p(n^{-1/2}).
%n^{1/2} \left (\tilde \alpha^* - \alpha^*\right ) \to ???
%\]
%as $n\to \infty$ and for  some constant  $\sigma_\alpha^2 >0$ depending on $\alpha$. This implies in particular $\tilde  \alpha^* - \alpha^* = O_p(n^{-1/2})$. 
%\end{lemma}
%\begin{proof}
%Following the proof of Lemma \ref{lem:alpha:star:parametric:rate:LC}, it is only left to show $\int \tilde \alpha^j(x_j) \hat V^j_{j,0}(x_j) \mathrm dx_j = o_P(h^2)$. 
%\quad  \\
%\begin{center}  \textbf{missing: $\int \tilde \alpha^j(x_j) \hat V^j_{j,0}(x_j) \mathrm dx_j $}  \end{center}
%\end{proof}

\begin{proof}[Proof of Theorem \ref{thm:main:theorem:ll}]
To apply Propositions \ref{thm:conv:of:backfitting:ll}--\ref{thm:bias:part:ll}, we have to prove that Assumptions A1--A5 imply B1-B6, B8, B9. The proof is analogous to the proof of Theorem \ref{thm:main:theorem:lc} and the assumptions can be shown in a similar way. 

We now focus on the variance and bias part
\begin{align*} 
\left( \begin{matrix} \hat \alpha_j^A(x_j)  \\ \hat \alpha^{j,A} (x_j)   \end{matrix} \right) &=  \hat M_j(x_j)^{-1}    \frac 1 n \sum_{i=1}^n \int \left( \begin{matrix} 1 \\ h^{-1} (x_j - X_{ij}(s))\end{matrix}\right)  k_h(x_j,X_{ij}(s))  \mathrm d M_i(s), \\
\left( \begin{matrix} \hat \alpha_j^B(x_j)  \\ \hat \alpha^{j,B} (x_j) \end{matrix} \right) &=  \hat M_j(x_j)^{-1} \frac 1 n \sum_{i=1}^n \int  \left( \begin{matrix} 1 \\ h^{-1} (x_j - X_{ij}(s))\end{matrix}\right)  k_h(x_j,X_{ij}(s))  \mathrm d \Lambda_i(s). 
\end{align*}
Analogously to \eqref{eq:sup:Ih:Ej}--\eqref{eq:sup:01:Ejk}, we show uniform convergence of $\hat M_j(x_j)$ and $\hat S_{l,j}(x_l, x_j)$ to $M_j(x_j)$ and $S_{l,j}(x_l, x_j)$, respectively,  and then focus on 
\[\frac 1 n \sum_{i=1}^n \int \left( \begin{matrix} 1 \\ h^{-1} (x_j - X_{ij}(s))\end{matrix}\right)  k_h(x_j,X_{ij}(s)) \mathrm d M_i(s)\] 
for asymptotic normality and on
\[\frac 1 n \sum_{i=1}^n \int \left( \begin{matrix} 1 \\ h^{-1} (x_j - X_{ij}(s))\end{matrix}\right)  k_h(x_j,X_{ij}(s))  \mathrm d \Lambda_i(s)\] 
for a bias term. 

With $M_i$ being the same martingale as in the proof of Theorem \ref{thm:main:theorem:lc} occurring in the stochastic part, we get the same asymptotic variance $\sigma_j^2$. Moreover, Assumptions A6--A9 can be verified with the choices 
\begin{align*}
%W &= \left( \begin{matrix} 1 & 0 \\ 0 & \int u^2 k(u) \mathrm du \end{matrix} \right), \\ 
\Delta_n &= h^2, \\
a_n^* &= \alpha^*, \\
a_{n,j}(x_j) &= \alpha_j(x_j) + \frac{1}{2} h^2 \alpha_j''(x_j) \int u^2 k(u) \mathrm d u,  \\
a_{n}^j(x_j) &= h \alpha_j'(x_j) , \\
\beta(x) &= \sum_{j=1}^d \frac 1 2 \int u^2 k(u)\mathrm du \left [ \alpha_j''(x_j) - \int \alpha_j''(x_j) E_j(x_j) \mathrm dx_j\right], \\
\gamma_{n,j} &= \nu_{n,j} + \frac{h^2}{2}  \int u^2 k(u) \mathrm d u  \int \alpha_j''(x_j) E_j(x_j)  \mathrm d x_j,\\
\nu_{n,j} &= \int \int \alpha_j(x_j) k_h(x_j,u) E_j(u) \mathrm du \, \mathrm dx_j.
\end{align*}
%
%
%\begin{center} \textbf{B1', B2'  easy} \end{center}
%\begin{center} \textbf{show certain bits here of B3'--B6', B8', B9' ? } \end{center} 
%
%
%\begin{center} \textbf{ check $\beta$ !!!!}\end{center}
%\quad \\ \quad
%\begin{center} \textbf{check constant component $\tilde\alpha^*$: }\end{center}
%The constant term $\alpha^*$ is estimated by $\tilde \alpha^{*} = \frac { \sum_{i=1}^n  \int \mathrm dN_i(s) }{\sum_{i=1}^n \int Y_i(s) \mathrm  ds}  - \sum_{j=0}^d \left(\int  \hat V^j_{j,0}(x_j) \mathrm dx_j \right)^{-1} \int \tilde \alpha^j(x_j) \hat V^j_{j,0}(x_j) \mathrm dx_j $  as defined in equation \eqref{eq:alphastar:LL}. Here, the first summand is estimated at $n^{1/2}$-convergence rate due to Lemma \ref{lem:alpha:star:parametric:rate:LC} and the second summand is a normalization through the estimators $\tilde \alpha^j(x_j)$, $\hat V^j_{j,0}(x_j)$, $j=0,\dots,d$ themselves. 
\end{proof}

\subsection{Two-step smooth backfitting estimator} \label{sec:twostep}
The interpretation as a projection motivates two different ways to compute the smooth backfitting hazard estimator. 
For the minimisation over all additive hazard functions, we can either minimize directly or we first minimize over the subspace of all (unstructured) local polynomial functions of degree $p$ obtaining a solution $\hat \alpha_{pilot}$ from \eqref{eq:criterion} which is a non-additive estimator and then minimize the integrated squared errors between $\hat \alpha_{pilot}$ and all additive local polynomial functions of degree $p$: 
\begin{align}
\begin{split}  \label{eq:criterion:twostep}
 \argmin_{\substack{\alpha^* \in \mathbb R, \\ \alpha^{(l)}_j:\mathbb R \to \mathbb R,   \\ j = 0,\dots,d \\ l= 0,\dots,p}} \sum_{i=1}^n \int \int  &\left \{ \hat \alpha_{pilot}(x)  -  \Big[ \alpha^* + \alpha_0(t) + \alpha_1(z_1) + \dots \alpha_d(z_d)   \right.  \\ 
&\ +  \left. \alpha_0^{(p)}(x_0)\left(\frac{x_0 - X_{i0}(s)}{h}\right)^p+ \dots  + \alpha^{(p)}_d(x_d)\left(\frac{x_d - X_{id}(s)}{h}\right)^p \Big] \right \}^2  \\
& \times K_h(x-X_i(s)) Y_i(s) \mathrm ds \,  \mathrm d\nu(x). 
\end{split} 
\end{align}

We want to emphasize that the estimator we obtain via direct minimisation \eqref{eq:criterion:LC} or \eqref{eq:criterion:LL}, respectively,  and the one obtained through the two-step minimisation \eqref{eq:criterion:twostep} are identical. %Details on this are given in Appendix \ref{appendix:projection}. 
%The two-step approach for a local constant estimator is illustrated in Section \ref{sec:twostep} and the direct one for a local linear estimator is given in Section \ref{sec:direct}. Both approaches have certain advantages that are illustrated in the sequel.  

In the following, we want to illustrate how the estimator can be obtained from an unstructured hazard estimator. Although we don't make use of it, this representation enables us to derive the asymptotic theory for the final estimator making use of the known asymptotic behavior of the established unstructured local constant %and local linear hazard estimators $\hat\alpha^{LC}$ and $\hat\alpha^{LL}$, respectively, 
which is defined below. 
Moreover, the derivation is less technical and easier to follow and the implementation is more straightforward. 

%\subsubsection{Local constant two-step estimator}  
Let $\hat \alpha$ be the unstructured local constant pilot estimator, $\hat \alpha^{LC}$ defined in Section \ref{sec:loc:const}. Then, for a weighting $w$, the local constant smooth backfitting estimator $\bar\alpha$ can be equivalently defined as
\[
\min_{\bar\alpha} \int_{\mathcal{X}}  \left (\hat \alpha(x) - [\bar \alpha^* + \sum_{j=0}^d \bar \alpha_j(x_j)]\right)^2 w(x)\mathrm   dx
.\]

Analogously, for $p=1$ we get the local linear estimator $\hat \alpha^{LL}(x)= \hat O^{LL}(x) / \hat E^{LL}(x) $ for $x\in\mathcal X$ from equation  \eqref{eq:criterion}, which is defined through 
\begin{align*}
\hat O^{LL}(x) &= \frac 1 n \sum_{i=1}^n \int \{1-(x-X_i(s))D(x)^{-1}c_1(x)\} K_h(x,X_i(s))\mathrm dN_i(s), \\
\hat E^{LL}(x) &= \frac 1 n \sum_{i=1}^n \int \{1-(x-X_i(s))D(x)^{-1}c_1(x)\} K_h(x,X_i(s)) Y _i (s) \mathrm ds, 
\end{align*}
where $c_j(x) = n^{-1} \sum_{i=1}^n \int K_h(x,X_i(s))(x_j-X_{ij}(s)) Y_i(s) \mathrm ds$ and for the $(d+1)\times(d+1)$-matrix $D(x)= [d_{jk}(x)]_{jk}$ with $d_{jk}(x) = \frac 1 n \sum_{i=1}^n \int K_h(x,X_i(s))(x_j-X_{ij}(s))(x_k-X_{ik}(s))Y_i(s)\mathrm ds$. 

Note that the matrix $D$ is not necessarily regular for $d>2$ %, see e.g.\ \cite{Mammen:etal:99} 
and hence the existence of $D^{-1}$ and the existence of $\hat \alpha ^{LL}$ are not guaranteed for $d>2$. 

In contrast to the local linear estimator, the local constant estimator $\hat \alpha^{LC}$ is always well defined independent of the dimension $d$.

\section{Fitted values from the multiplicative model}

In this section we show the fitted values from the local constant multiplicative smooth backfitting model \cite{hiabu2021smooth} applied to the TRACE study data application from Section \ref{sec:tracestudy}.
The fit for the risk in the first three months is given in Figure 
\ref{fig:app:mult1} and the fit for the risk conditional on surviving the first three months is given in
\ref{fig:app:mult2}.

 \begin{figure}
 \include{early_mult.tex}
 \caption{
 Local constant multiplicative smooth backfitting  fit of $(\alpha_0,\alpha_1,\alpha_2)$ conditional on surviving the first three months
 for two different strata depending on the value of vf. 
 }
 \label{fig:app:mult1}
 \end{figure}

 \begin{figure}
 \include{late_mult.tex}
 \caption{
Local constant multiplicative smooth backfitting fit of $(\alpha_0,\alpha_1,\alpha_2)$ conditional on surviving the first three months
 for two different strata depending on the value of vf. 
 }
 \label{fig:app:mult2}
 \end{figure}
 
\section*{Acknowledgment}
This work was supported by the Deutsche Forschungsgemeinschaft (DFG) through the Research Training Group RTG 1953. 

\bibliographystyle{abbrvnat}
\bibliography{references}
\end{document}